\documentclass[11pt,reqno]{amsproc}

\title[]{Compressible fluids and active potentials}
\author{Peter Constantin, Theodore D. Drivas, Huy Q. Nguyen, and Federico Pasqualotto}
\address{Department of Mathematics, Princeton University, Princeton, NJ 08544}
\email{const@math.princeton.edu}
\address{Department of Mathematics, Princeton University, Princeton, NJ 08544}
\email{tdrivas@math.princeton.edu}
\address{Department of Mathematics, Princeton University, Princeton, NJ 08544}
\email{qn@math.princeton.edu}
\address{Department of Mathematics, Princeton University, Princeton, NJ 08544 \\ \newline
 and DPMMS, University of Cambridge, Cambridge CB3~0WA, United Kingdom}
\email{fp2@math.princeton.edu}

\usepackage[margin=1in]{geometry}
\usepackage{amsmath, amsthm, amssymb}
\usepackage{mathrsfs}
\usepackage{graphicx}

\usepackage{times}
\usepackage[dvipsnames]{color}

\usepackage[colorlinks=true, pdfstartview=FitV, linkcolor=blue, citecolor=black, urlcolor=blue]{hyperref}

\newcommand{\bq}{\begin{equation}}
\newcommand{\eq}{\end{equation}}
\newcommand{\bqa}{\begin{eqnarray*}}
\newcommand{\eqa}{\end{eqnarray*}}
\newcommand{\Rr}{\mathbb{R}}

\newcommand{\T}{\mathbb{T}}

\newcommand{\la}{\label}

\newcommand{\be}{\begin{equation}}
\newcommand{\ee}{\end{equation}}
\newcommand{\ba}{\begin{array}{l}}
\newcommand{\ea}{\end{array}}

\newcommand{\red}[1]{\textcolor{red}{#1}}

\theoremstyle{plain}
\newtheorem{theo}{Theorem}[section]
\newtheorem{prop}[theo]{Proposition}
\newtheorem{lemm}[theo]{Lemma}

\theoremstyle{definition}
\newtheorem{rema}[theo]{Remark}

\DeclareSymbolFont{pletters}{OT1}{cmr}{m}{sl}
\DeclareMathSymbol{s}{\mathalpha}{pletters}{`s}

\def\la{\left\lvert}

\def\le{\leq}

\def\mez{\frac{1}{2}}

\def\ra{\right\rvert}

\def\tdm{\frac{3}{2}}

\newcommand{\de}{\text{d}}

\def\p{\partial}

\def\cl{ \underline{\rho} }

\numberwithin{equation}{section}

\begin{document}

\begin{abstract}
We consider a class of one dimensional compressible systems with degenerate diffusion coefficients. We establish the fact that the solutions remain smooth as long as the diffusion coefficients do not vanish, and give local and global existence results. The models include the barotropic compressible Navier-Stokes equations, shallow water systems and the lubrication approximation of slender jets. In all these models the momentum equation is forced by the gradient of a solution-dependent potential: the active potential. The method of proof uses the Bresch-Desjardins entropy and the analysis of the evolution of the active potential.\\
\phantom{0}\hfill \today
\end{abstract}



\maketitle

\section{Introduction}

We consider a class of compressible fluid models in one space dimension with periodic boundary conditions:
\begin{align}
	&\partial_t \rho + \partial_x (u \rho ) = 0, \label{eq:mass} \\
	&\partial_t (\rho u) + \partial_x (\rho u^2) =- \partial_x p(\rho) + \partial_x(\mu(\rho) \partial_xu)+\rho f\label{eq:mom}, \\
	&(\rho,u)|_{t = 0} = (\rho_0, u_0) \label{eq:initv}
\end{align}
with constitutive laws given by
\be\label{EOS}
p(\rho) = c_p  \rho^{\gamma}, \qquad \mu(\rho)=c_\mu \rho^\alpha, \qquad  c_p \neq 0, \ c_\mu>0.
\ee

Among these models are the one-dimensional barotropic compressible Navier-Stokes equations.  
In this description, $\rho$ is the mass density, $u$ is the fluid velocity, and $p(\rho)$, $\mu(\rho)$ 
are the fluid pressure and dynamic viscosity respectively.  These are given by physical equations of state \eqref{EOS}.
For such systems, the specific heat at constant pressure is positive $c_p >0$ so that $p(\rho)$ is non-negative.
The viscosity is also assumed non-negative $c_\mu>0$ but may be degenerate in the sense that it
vanishes for $\rho=0$.  

Although the eqns. \eqref{eq:mass}--\eqref{eq:initv} describe cases of compressible Navier-Stokes equations,
they serve also as models for a number of other physical systems if the basic variables and constitutive laws are appropriately defined. For example, a model for  viscous incompressible
motion of shallow water waves \cite{perthame2001,Marche2007} reads
\begin{align}
	&\partial_t h + \partial_x (u h) = 0, \label{sww:h} \\
	&\partial_t(hu) + \partial_x(hu^2) + \frac{g}{2}\partial_x h^2= 4\nu\partial_x (h \partial_x u)+hf\label{sww:v}
\end{align}
where 
 \begin{itemize}
 \item $h$ and $u$ represent respectively the surface height and fluid velocity,
 \item $g$ is gravity,
 \item $\nu>0$ is the kinematic viscosity,
 \item $f$ is the external force. 
 \end{itemize}
These equations are a special case of equations \eqref{eq:mass}-\eqref{eq:mom} with 
\[
p(\rho)=\frac{g}{2}\rho^2\quad\text{and}\quad \mu(\rho)=4\nu\rho.
\]
Equations \eqref{eq:mass}--\eqref{eq:initv} also appear in the theory of drop formation as the slender jet equations \cite{JD94, JF15}:
\begin{align}
	&\partial_t h + u \p_xh =-\mez\p_xu h , \label{jet:h} \\
	&\partial_tu + u\p_xu + \gamma\partial_x (\frac{1}{h}) = 3\nu\frac{\partial_x(h^2\partial_xu)}{h^2}-g\label{jet:v},
\end{align}
 where 
 \begin{itemize}
 \item $h$ and $u$ represent respectively the neck radius and velocity of the jet,
 \item $\gamma>0$ is the surface tension coefficient,
 \item $\nu>0$ is the kinematic viscosity,
 \item $g>0$ is gravity. 
 \end{itemize}
These equations arise as a reduction of the axisymmetric incompressible Navier-Stokes equations in two spatial dimensions  governing a thin liquid thread with a moving boundary.  Via the change of variables $\rho=h^2$, equations \eqref{jet:h}-\eqref{jet:v} become equations \eqref{eq:mass}-\eqref{eq:mom} with 
\[
p(\rho)=-\gamma\sqrt \rho\quad\text{and}\quad \mu(\rho)=3\nu\rho.
\]
 Note that here 
 the ``pressure" that appears is non-positive in contrast with the Navier-Stokes descriptions.

 In all the settings above, the one-dimensional equations \eqref{eq:mass}--\eqref{eq:initv} are approximate
 models of the underlying physical processes, whose quality may vary depending on the situation.  
As models for dissipative molecular fluids, they are not known to arise as an effective description by a controlled hydrodynamic limit and do not conserve total energy.  See Section A and Appendix B of \cite{ED18} for an extended discussion. 
Of course, they could be valid descriptions of fluid systems in other situations than these, as is the case of the shallow water and slender jet.   Moreover, J. Eggers has argued that the slender
 jet equations described above become an exact description asymptotically close to drop pinch--off, justifying the use of the model \eqref{jet:h}, \eqref{jet:v}  in that context.

 Four theorems are proved.  The first result, Theorem \ref{theo:cont}, provides a blowup criterion for equations \eqref{eq:mass}--\eqref{eq:initv} with a wide range of constitutive pressure and viscosity laws \eqref{EOS}.  In what follows, we denote by $\T$ the interval $(0, 1]$ with periodic boundary conditions.

\begin{theo}\label{theo:cont} 
	Assume any of the following three conditions
	\begin{enumerate}
		\item[(i)] $c_p >0$ and $\alpha> \mez$, $\gamma\neq 1,\   \gamma\geq \alpha-\mez$,
		\item[(ii)] $c_p <0$ and  $\mez<\alpha\le \tdm$, $\gamma<1$, $0<\gamma\leq \alpha$,
		\item[(iii)] $c_p>0$ and $\gamma>1$, $\alpha\ge 0$.
	\end{enumerate}
	Let $k\geq 3$ and assume further that
	\begin{equation*}
	 f\in L^2(0,T; H^{k-1}(\T))  \quad \text{for all} \quad  T > 0.
	\end{equation*}
	If $(\rho, u)$ is a solution of \eqref{eq:mass}-\eqref{eq:initv} on $[0, T^*)$ such that 
	\be\label{class}
	\rho \in C(0, T; H^k(\T)),\quad u\in C(0, T; H^k(\T))\cap L^2(0, T; H^{k+1}(\T)), \quad\forall T\in (0, T^*)
	\ee
	and
	\[
	\inf_{t\in [0, T^*)} \min_{x\in \T}\rho (x, t)>0,
	\]
	then $(\rho, u)$ satisfies
	\be
	\sup_{T\in [0,T^*)} \Vert \rho\Vert_{L^\infty(0, T;  H^k)}+\sup_{T\in [0,T^*)}\Vert u\Vert_{L^\infty(0, T; H^k)}+\sup_{T\in [0,T^*)}\Vert u\Vert_{L^2(0, T; H^{k+1})}<\infty
	\ee
	and can be continued in the class \eqref{class} past $T^*$.
\end{theo}
Theorem \ref{theo:cont} says that the only possible way for a singularity to form starting from smooth data is if the density becomes zero somewhere in the domain. This applies in particular to the viscous shallow water wave equations \eqref{sww:h}-\eqref{sww:v}.  In the slender jet equations \eqref{jet:h}-\eqref{jet:v} which model incompressible fluid drop formation, this says that singularities can only form at the onset of drop break-off.  This answers a conjecture of P. Constantin recorded in \cite{JD94}. 

\begin{rema}
The conclusions of Theorem \ref{theo:cont} hold whenever  an upper bound on the density of the form \eqref{rhoBnd2} exists, possibly dependent on the minimum density $\underline{\rho}$.  Under any of the conditions (i), (ii), (iii) of the Theorem, we produce such a bound.  However, it seems unlikely that (i)--(iii)  are fundamental restrictions, and the result should hold over  larger range conditions.
\end{rema}

\begin{rema}
\cite{Hoff2001} proved that weak solutions of 1D compressible Navier-Stokes equations with constant viscosity do not exhibit vacuum states in finite time provided no vacuum states are present initially.
\end{rema}
\begin{rema}
Local well-posedness of \eqref{eq:mass}--\eqref{eq:initv}  in the class \eqref{class} is established in Proposition \ref{prop:local} of the Appendix \ref{appendix:Loc} for arbitrary smooth $p(\rho)$ and smooth non-negative $\mu(\rho)$. This covers the special case of power law equations of state \eqref{EOS} in the entire parameters range in Theorem \ref{theo:cont}. Local existence of strong solution for 2D shallow water equations can be found in \cite{AnTon, Li2017}. We also refer to \cite{Matsumura1980, Matsumura1983} for classical results regarding equations of compressible viscous and heat-conductive fluids with constant viscosity. 
\end{rema}
Our next two theorems concern the long-time existence and persistence of regularity.  Theorem \ref{theo:global}
establishes global existence for arbitrarily large data, within a range of pressure and viscosity of the form \eqref{EOS}.
\begin{theo}\label{theo:global}
	Assume
	$$ c_p >0,  \quad\alpha\in (\mez,1],  \text{ and }\quad \gamma\geq 2\alpha.
	$$  
	Let $k\ge 3$ be an integer and let $\rho_0$ and $u_0$ belong to $H^k(\T)$ such that $\rho_0(x)>0$ for all $x\in \T$.
	Assume further that
	\begin{equation*}
	 f\in L^2(0,T; H^{k-1}(\T))  \quad \text{for all} \quad  T > 0.
	\end{equation*}
	Then there exists a unique global solution $(\rho, u)$ to \eqref{eq:mass}-\eqref{eq:initv} such that 
	\[
	\rho\in C(0, T; H^k(\T)),\quad u\in C(0, T; H^k(\T))\cap L^2(0, T; H^{k+1}(\T))
	\]
	for all $T>0$, and $\rho(x, t)>0$ for all $(x, t)\in \T\times \Rr^+$.
\end{theo}
 This result applies to the viscous shallow water equations \eqref{sww:h}-\eqref{sww:v}, giving an alternative proof to that of \cite{H14}. Let us note that \cite{H14} assumes only $H^1$ regularity of initial data. Moreover, Theorem \ref{theo:global} allows for more singular density dependence of the viscosity than in \cite{MV06}, which considers the case of $\alpha<\mez$ and $\gamma>1$. In two dimensions, global stability of constant solutions to shallow water equations was proved in \cite{Sundbye1996, Sundbye1998,Kloeden1985}.

For more degenerate viscosity $\rho^\alpha$ allowing $\alpha>1$, we prove global existence for a class of large initial data.
\begin{theo}\label{theo:global2}
Assume that $c_p>0$ and either 
\begin{align}
\label{assumThm1.2}
&\alpha>\mez, \quad \gamma\in [\alpha,\alpha+1], \quad \gamma \neq 1\quad\text{or}\\ \label{assumThm1.2:2}
&\alpha\ge 0,\quad \gamma\in [\alpha,\alpha+1],\quad \gamma>1.
\end{align}
Assume further that
\[
f(x, t)=f(t)\in L^2((0, T))\quad\forall T>0.
\]
Let $k\ge 4$ be an integer and let $u_0$ and $\rho_0$ belong to $H^k(\T)$ such that $\rho_0(x)>0$ for all $x\in \T$ and 
\be\label{initialDataCond}
\partial_x u_0(x)\le  \frac{c_p }{c_\mu} \rho_0(x)^{\gamma-\alpha} \quad\forall x\in \T.
\ee
Then there exists a unique global solution $(\rho, u)$ to \eqref{eq:mass}-\eqref{eq:initv} such that 
\[
\rho\in C(0, T; H^k(\T)),\quad u\in C(0, T; H^k(\T))\cap L^2(0, T; H^{k+1}(\T))
\]
for all $T>0$, and $\rho(x, t)>0$ for all $(x, t)\in \T\times \Rr^+$.
\end{theo}
\begin{rema}
We note that  \eqref{initialDataCond} does not impose any smallness conditions on the initial data. The unique global solution in Theorem \ref{theo:global} satisfies
\[
\partial_x u (x, t)\le  \frac{c_p }{c_\mu} \rho(x,t)^{\gamma-\alpha} 
\]
for all $(x, t)\in \T\times \Rr^+$.  Moreover, the proof provides a lower bound for the minimum of density $\rho$, see \eqref{firstmin} and \eqref{secondmin},
\[
\min_{x\in \T}\rho(x, t)\geq
\begin{cases}
  \left(\rho_m(0)^{\alpha -\gamma} + t\frac{c_p }{c_\mu}(\gamma-\alpha)\right)^{\frac{-1}{\gamma-\alpha}}\quad\text{when}~\gamma>\alpha,\\
  \rho_m(0)\exp\left(-t \frac{c_p}{c_\mu}\right)\quad\qquad\qquad\quad\text{when}~\gamma=\alpha.
  \end{cases}
\]
\end{rema}
Our last theorem establishes a bound on the time-averaged maximum density for a certain range of parameters assuming mean zero forcing.
\begin{theo}\label{densBndProp}
	Assume that $(\rho, u)$ is a sufficiently smooth solution to the system~\eqref{eq:mass}--\eqref{eq:initv} on $[0,T^*)$.
	Assume that 
	\begin{equation}\label{zeromean}
	f = \p_x g
	\end{equation}
	for some periodic function  $g$ satisfying 
	\[
	g\in L^\infty(0,T^*; L^\infty(\T)),\quad \text{and}\quad\p_xg, \p_tg \in L^\infty(0,T^*; L^\infty(\T)).
	\]
	Let us also assume that
	\begin{equation*}
	\alpha \geq 1/2,\quad \gamma \in [\max\{2-\alpha,\alpha\}, \alpha + 1], \quad \text{and} \quad c_p, c_\mu > 0.
	\end{equation*}
	Then, we have the following bound
	\begin{equation}
	\frac 1 T \int_0^T \Vert \rho(\cdot,t)\Vert_{L^\infty(\T)}\de t \leq C_1 + \frac{1}{T} C_2,
	\end{equation}
	where $C_1$ and $C_2$ 
	are defined in equation~\eqref{eq:m4m5}. In particular, $C_1$ depends only on $c_\mu$, $c_p$, $\alpha$, $\gamma$, $\Vert\rho_0 \Vert_{L^1}$, $\Vert \p_xg\Vert_{L^\infty(0, T; L^\infty)}$, and $\Vert \p_t g \Vert_{L^\infty(0,T; L^\infty)}$, whereas $C_2$ depends only on $c_\mu$, $c_p$, $\gamma$, $\alpha$, $\Vert\rho_0 \Vert_{L^\infty}$, $\Vert\rho^{-1}_0 \Vert_{L^\infty}$, $\Vert u_0\Vert_{L^2}$, $\Vert \p_x \rho_0 \Vert_{L^2}$, and $\Vert g \Vert_{L^\infty(0,T; L^\infty)}$. Consequently, if $T^*=\infty$ then 
	\begin{equation}\label{long-timebnd}
	\limsup_{T\to \infty}\frac 1 T \int_0^T \Vert \rho(\cdot,t)\Vert_{L^\infty(\T)}\de t \leq C_3
	\end{equation}
	where $C_3$ depends only on $c_\mu$, $c_p$, $\alpha$, $\gamma$, $\Vert\rho_0 \Vert_{L^1}$, $\Vert \p_xg\Vert_{L^\infty(0, \infty; L^\infty)}$, and $\Vert \p_t g \Vert_{L^\infty(0,\infty; L^\infty)}$.
\end{theo}
Theorem \ref{densBndProp} applies for the viscous shallow water wave system \eqref{sww:h},\eqref{sww:v} for which global existence is established by Theorem \ref{theo:global}.   The interpretation of the bound \eqref{long-timebnd} with $h\equiv \rho$ is that long-time average of the maximum surface height remains bounded, showing that, on average, no extreme events can develop.
\begin{rema}
Modulo technical conditions, Theorems \ref{theo:cont}, \ref{theo:global}, \ref{theo:global2} and \ref{densBndProp} should hold for more general constitutive laws $\mu(\rho)$ and $p(\rho)$ that behave asymptotically when $\rho\to 0$ as $c_\mu\rho^\alpha$ and $c_p\rho^\gamma$ respectively. The high regularity of initial data  in the above Theorems is assumed to apply maximum principles straightforwardly. By appealing to more refined maximum principles, the regularity of initial data can be reduced. 
\end{rema}
The proofs are based on use of the Bresch-Desjardins entropy and analysis of the evolution of the active potential $w$. This object is the potential in the momentum equation \eqref{eq:mom}: its gradient is the force
\begin{align}
	\rho D_tu  = \partial_x w.
\end{align}
The potential 
\[
w=- p(\rho) +\mu(\rho) \partial_xu.
\]
is unknown and combines the viscous stress with the pressure.  
As $w$ depends on the unknowns and in turn determines their evolution, we refer to it as an \emph{active potential}.  Remarkably, $w$ satisfies a {\it forced quadratic heat equation with linear drift and less degenerate diffusion} with the new dissipation term $\frac{\mu(\rho)}{\rho}\p_x^2w$. The active potential $w$ contains one derivative of $u$ and no derivative of $\rho$. On one hand, energy estimates for the coupled system of $\rho$ and $w$ allow us to control all the high Sobolev regularity of $\rho$ and $u$ as long as $\rho$ is positive, leading to the proof of Theorem \ref{theo:cont}. On the other hand, the heat equation for $w$ satisfies a maximum principle which enables us to obtain global regular solutions for a class of large data when the viscosity is strongly degenerate as in Theorem \ref{theo:global2}. 

The fact that the active potential solves a nondegenerate evolution with a maximum principle was observed in \cite{CENV} in the context of a 1D Hele Shaw model, where it served a similar role. The effective viscous flux used in \cite{Hoff1995} and \cite{Lions} is an active potential: there it was used by inverting the elliptic (nondegenerate) equation it solves at each fixed time. 

\section{A priori estimates: mass, energy and Bresch-Desjardins's entropy}

Assume that $(\rho, u)$ is a solution of \eqref{eq:mass}-\eqref{eq:initv} on the time interval $[0, T^*)$ such that 
\[
\rho\in C(0, T; H^3),\quad u\in C(0, T; H^3)\cap L^2(0, T; H^4)
\]
for any $T<T^*$ and
\bq\label{c0}
 \cl :=\inf_{t\in [0, T^*)}\min_{x\in\T}\rho(x, t)>0.
\eq
In what follows we denote by $M(\cdot,\  \cdots, \cdot)$ a positive function that is increasing in each argument.

First, from the continuity equation \eqref{eq:mass}, total mass is conserved:
\be\label{massConsbnd}
\Vert \rho(\cdot, t)\Vert_{L^1(\T)}=\Vert \rho_0\Vert_{L^1(\T)}.
\ee
We have the following standard energy balance:
\begin{lemm}[Energy Balance]\label{energy:0}
Let $\bar \rho \geq 0$, and
\be\label{piEqn}
e:= \frac{1}{2} \rho |u |^2 + \pi (\rho), \qquad 
\pi(\rho)= \rho \int_{{\bar \rho}}^\rho \frac{p(s)}{s^2} ds.
\ee
Then, the balance
\begin{equation} \label{energyBalance}
	\frac{\de}{\de t}\int_{\T}  e(x,t) \de x  = -\int_{\T} \mu(\rho) |\partial_x u|^2\de x+\int_\T f\rho u\de x
\end{equation}
holds for any $t\in [0, T^*)$.
\end{lemm}

Using the equation of state for the density \eqref{EOS} and recalling that $\bar{\rho}\geq 0$ is an arbitrary constant that we are free to fix, we have an explicit formula for $\pi(\rho)$ from \eqref{piEqn}
\bq\label{formula:pi}
\pi(\rho)= c_p \rho \int_{\bar{\rho}}^\rho s^{\gamma-2}ds =\begin{cases}
\frac{c_p }{\gamma-1}\rho^{\gamma}  &\quad \gamma >1,\ \bar{\rho}=0\quad \text{or} \quad \gamma \in(0,1), \ \bar{\rho}=\infty,\\
c_p \rho \log(\rho)  &\quad \gamma = 1, \ \bar \rho =1.
\end{cases}
\eq
Note that 	the function $\pi$ satisfies
	$$
	\pi''(\rho) = \frac{p'(\rho)}{\rho}.
	$$	
\begin{lemm}\label{PiCont:0}
1. If $\gamma\in (1,\infty)$ and $c_p >0$, then $\pi(\rho)\geq 0$ and
\be\label{Picont1}
\|e\|_{L^\infty (0,T;L^1)} +\Vert \mu(\rho)|\p_x\rho|^2\Vert_{L^1(0, T; L^1)}\leq \left(\Vert e(\cdot, 0)\Vert_{L^1}+\Vert f\Vert_{L^2(0, T; L^\infty)}^2\Vert \rho_0\Vert_{L^1(\T)}\right)\exp(2T).
\eq
2.  If $\gamma\in(0,1)$ and $c_p \neq 0$, then
\be\label{bound:pi:neg}
\int_{\T} | \pi(\rho)|dx\leq  \left|\frac{c_p }{\gamma-1}\right|\int (\rho_0+1) dx
\ee
and there exists a positive constant $C=C(\gamma, \alpha, c_p, c_\mu)$ such that
\bq\label{bound:rhou2}
\begin{aligned}
& \Vert \rho u^2\Vert_{L^\infty(0, T; L^1)}+\Vert \mu(\rho)|\p_x\rho|^2\Vert_{L^1(0, T; L^1)}\\
&\le \left(\Vert\rho_0 u_0^2\Vert_{L^1(\T)}+C\big(1+\Vert f\Vert_{L^2(0, T; L^\infty)}^2\big)\big(1+\Vert \rho_0\Vert_{L^1(\T)}\big)\right)\exp(T).
 \end{aligned}
 \eq
 \end{lemm}
\begin{proof}
First, using the mass conservation \eqref{massConsbnd} we bound
\bq\label{estf:energy}
\begin{aligned}
\int_\T f\rho u\de x&\le \mez\int_\T f^2\rho+\int_\T \mez \rho u^2\\
&\le \Vert f\Vert_{L^\infty(\T)}^2\int_\T \rho+\int_\T \mez\rho u^2\\
&\le \Vert f\Vert_{L^\infty(\T)}^2\Vert\rho_0\Vert_{L^1(\T)}+\int_\T \mez\rho u^2.
\end{aligned}
\eq
1. If $\gamma\in(1,\infty)$ and $c_p >0$, then we have $\pi(\rho)\geq0$. It then follows  from \eqref{estf:energy} that
\bq\label{estf:energy1} 
\int_\T f\rho u\de x\le \Vert f\Vert_{L^\infty(\T)}^2\Vert \rho_0\Vert_{L^1(\T)}+\int_\T e(x, t)\de x.
\eq
 Ignoring the  first term on the right hand side of \eqref{energyBalance}, then using \eqref{estf:energy1} and Gr\"onwall's lemma we obtain
\bq\label{bound:e}
\|e\|_{L^\infty (0,T;L^1)} \leq \left(\Vert e(\cdot, 0)\Vert_{L^1}+\Vert f\Vert_{L^2(0, T; L^\infty)}^2\Vert \rho_0\Vert_{L^1(\T)}\right)\exp(T).
\eq
Next, we integrate \eqref{energyBalance} in time and use \eqref{estf:energy1}, \eqref{bound:e} together with the fact that $e(x, t)\ge 0$ to get
\begin{align*}
\Vert \mu(\rho)|\p_x\rho|^2\Vert_{L^1(0, T; L^1)}&\le \Vert e(\cdot, 0)\Vert_{L^1}+\Vert f\Vert_{L^2(0, T; L^\infty)}^2\Vert \rho_0\Vert_{L^1(\T)}+T\|e\|_{L^\infty (0,T;L^1)} \\
&\le \left(\Vert e(\cdot, 0)\Vert_{L^1}+\Vert f\Vert_{L^2(0, T; L^\infty)}^2\Vert \rho_0\Vert_{L^1(\T)}\right)(1+T)\exp(T)\\
&\le \left(\Vert e(\cdot, 0)\Vert_{L^1}+\Vert f\Vert_{L^2(0, T; L^\infty)}^2\Vert \rho_0\Vert_{L^1(\T)}\right)\exp(2T).
\end{align*}
2. If $\gamma\in (0,1)$ then
\bq\label{bound:pi:neg1}
\int_{\T} | \pi(\rho)|dx\leq  \left|\frac{c_p }{\gamma-1}\right|\int (\rho(t)+1) dx \leq  \left|\frac{c_p }{\gamma-1}\right|\int (\rho_0+1) dx
\eq
where we used the fact that  $\rho^\gamma\leq \max\{1,\rho\}$ together with the mass conservation \eqref{eq:mass}. Ignoring the first term on the right hand side of \eqref{energyBalance}  and using \eqref{bound:pi:neg1}, \eqref{estf:energy} we find
\begin{align*}
 \int_\T  \mez\rho u^2(x, t) \de x&\le\int_\T \mez\rho_0 u_0^2\de x+\int_\T \pi(\rho_0(x))\de x-\int_\T \pi(\rho(x, t))\de x+\int_0^t\int_\T f\rho u(x, s)\de x\de s\\
&\le \int_\T\mez \rho_0 u_0^2\de x+C(\Vert\rho_0\Vert_{L^1(\T)}+1)+ \Vert f(t)\Vert_{L^\infty(\T)}^2\Vert\rho_0\Vert_{L^1(\T)}+\int_0^t\int_\T \mez\rho u^2(x, s)\de x\de s
\end{align*}
for some positive constant $C=C(\gamma, \alpha, c_p, c_\mu)$.
Gr\"onwall's lemma then yields
\bq\label{rhou2:neg}
 \Vert \rho u^2\Vert_{L^\infty(0, T; L^1)}\le \left(\Vert\rho_0 u_0^2\Vert_{L^1(\T)}+C\big(1+\Vert f\Vert_{L^2(0, T; L^\infty)}^2\big)\big(1+\Vert \rho_0\Vert_{L^1(\T)}\big)\right)\exp(T).
 \eq
Again, we integrate \eqref{energyBalance} in time and use \eqref{estf:energy}, \eqref{rhou2:neg}, \eqref{bound:pi:neg1} to arrive at
\[
\Vert \mu(\rho)|\p_x\rho|^2\Vert_{L^1(0, T; L^1)}\le \left(\Vert\rho_0 u_0^2\Vert_{L^1(\T)}+C\big(1+\Vert f\Vert_{L^2(0, T; L^\infty)}^2\big)\big(1+\Vert \rho_0\Vert_{L^1(\T)}\big)\right)\exp(2T).
 \]
\end{proof}
If either  $\gamma\in (1,\infty)$ and $c_p >0$ or $\gamma\in (0,1)$ and $c_p \neq0$,  it follows from \eqref{formula:pi}-\eqref{bound:rhou2} that 
\begin{align} \label{enCont1}
\|\sqrt{\rho} u\|_{L^\infty (0,T;L^2)}  &\leq  M(E_0, \Vert f\Vert_{L^2(0, T; L^\infty)}, T) ,\\ \label{enCont2}
\|\rho^{\frac{\alpha}{2}} \partial_x u \|_{L^2 (0,T;L^2)}& \le M(E_0, \Vert f\Vert_{L^2(0, T; L^\infty)}, T),\\
\|\rho\|_{L^\infty (0,T;L^{\max\{1,\gamma\}})}&\leq  M(E_0, \Vert f\Vert_{L^2(0, T; L^\infty)}, T)\label{enCond3}
\end{align}
where
\bq
E_0:=\|\rho_0 u^2_0\|_{L^1(\T)}+\Vert \rho_0^\gamma\Vert_{L^1(\T)}+\Vert \rho_0\Vert_{L^1(\T)}.
\eq

\begin{lemm}[Bresch-Desjardins's Entropy \cite{Bresch2003}]\label{entropy:0}
Let 
\be
s:= \frac{\rho}{2} \left|u + \frac{\partial_x \rho}{\rho^2} \mu(\rho)\right|^2 + \pi (\rho).
\ee
 Then, the balance
\begin{equation} \label{entropyBalance}
	\frac{\de}{\de t}\int_{\T}  s(x,t) \de x  = -\int_{\T} |\partial_x \rho|^2 \mu(\rho) \frac{p'(\rho)}{\rho^2} dx+\int_\T  f\rho \big(u + \frac{\partial_x \rho}{\rho^2}\mu(\rho)\big)\de x\end{equation}
holds for any $t\in [0, T^*)$.
\end{lemm}
A proof of Lemma \ref{entropy:0} can be found in \cite{Bresch2003, Bresch2003b, BreschReview} and is given for completeness in the appendix. The first term on the right hand side of \eqref{entropyBalance} is negative whenever $c_p >0$ and positive whenever $c_p <0$.

\begin{lemm}\label{PiCont:1}
Define 
\bq
E_1:=E_0+\Vert \p_x(\rho_0^{\alpha-\mez})\Vert_{L^2(\T)}.
\eq
1. If $c_p >0$ and $\gamma\neq 1,\   \gamma\geq \alpha-\mez, \ \alpha> \mez$, then
\be\label{rhoBnd1}
\|\rho\|_{L^\infty(0,T;L^\infty)} \leq M(E_1, \Vert f\Vert_{L^2(0, T; L^\infty)}, T).
\ee
2.  If $c_p <0$ and  $0<\gamma\leq \alpha$, $\gamma<1$,  $\alpha\in (\mez, \tdm]$, then
\be\label{rhoBnd2}
\|\rho\|_{L^\infty(0,T;L^\infty)}  \leq M(E_1, \Vert f\Vert_{L^2(0, T; L^\infty)}, \frac{1}{ \cl }, T).
\ee
3.  Under the conditions of 1. or 2., we have
\be\label{rhoBnd3}
 \|\partial_x \rho\|_{L^\infty(0,T;L^2)}  \leq  M(E_1, \Vert f\Vert_{L^2(0, T; L^\infty)}, \frac{1}{ \cl }, T).
\ee

4. If $c_p>0$, $\gamma>1$ and $\alpha\ge 0$ then \eqref{rhoBnd2} and \eqref{rhoBnd3} hold.
\end{lemm}
\begin{rema}
The bound for \eqref{rhoBnd1} is independent of $ \cl $.  This fact will be important in the proof of Theorem~\ref{theo:global}.
\end{rema}
\begin{proof}
1. Since $c_p >0$, the first term on the right hand side of \eqref{entropyBalance}  is negative, and thus
\bq\label{bound:rhsentropy}
\begin{aligned}
\frac{\de}{\de t}\int_{\T}  s(x,t) \de x  &\le \int_\T  f\rho \big(u + \frac{\partial_x \rho}{\rho^2}\mu(\rho)\big)\de x\\
&\le\mez \int_\T  f^2\rho\de x+\int_\T  \mez\rho \big(u + \frac{\partial_x \rho}{\rho^2}\mu(\rho)\big)^2\de x\\
&\le\mez\Vert f(t)\Vert^2_{L^\infty(\T)}\Vert \rho_0\Vert_{L^1(\T)}+\int_\T  \mez\rho \big(u + \frac{\partial_x \rho}{\rho^2}\mu(\rho)\big)^2\de x.
\end{aligned}
\eq
When $\gamma>1$ we have $\pi(\rho)\ge 0$, hence $s>0$ and 
\[
\frac{\de}{\de t}\int_{\T}  s(x,t) \de x  \le\mez\Vert f(t)\Vert^2_{L^\infty(\T)}\Vert \rho_0\Vert_{L^1(\T)}+\int_\T s(x, t)\de x.
\]
Gr\"onwall's lemma then yields
\be\label{sGron}
\Vert s\Vert_{L^\infty(0, T; L^1)}\le \left(\Vert s(0, \cdot)\Vert_{L^1(\T)}+\Vert f\Vert^2_{L^2(0, T; L^\infty)}\Vert \rho_0\Vert_{L^1(\T)}\right)\exp(T).
\ee
We combine \eqref{sGron} with \eqref{enCont1} and the fact that
\bq\label{bound:s0}
\Vert s(0,\cdot)\Vert_{L^1(\T)}\le \Vert \rho_0u_0^2\Vert_{L^1(\T)}+\Vert \p_x(\rho_0^{\alpha-\mez})\Vert_{L^2(\T)}^2.
\eq
 In view of \eqref{enCont2}, this implies
\be\label{entinqe1}
\|\partial_x (\rho^{\alpha-\frac{1}{2}} ) \|_{L^\infty(0,T;L^2(\T))} \leq M(E_1, \Vert f\Vert_{L^2(0, T; L^\infty)}, T)
\ee
with 
\[
E_1=E_0+\Vert \p_x(\rho_0^{\alpha-\mez})\Vert_{L^2(\T)}.
\]
On the other hand, when $\gamma\in (0, 1)$ we write 
\[
\frac{\de}{\de t}\int_{\T}  \mez\rho \big(u + \frac{\partial_x \rho}{\rho^2}\mu(\rho)\big)^2 \de x\le \frac{\de}{\de t}\int_{\T}\pi(\rho(x, t))\de x+\mez\Vert f(t)\Vert^2_{L^\infty(\T)}\Vert \rho_0\Vert_{L^1(\T)}+\int_\T  \mez\rho \big(u + \frac{\partial_x \rho}{\rho^2}\mu(\rho)\big)^2\de x
\]
where we recall from \eqref{bound:pi:neg}
\be
\int_{\T} | \pi(\rho)|dx\leq  \left|\frac{c_p }{\gamma-1}\right|\int (\rho_0+1) dx.
\ee
It follows from Gr\"onwall's lemma that
\begin{align*}
\sup_{t\in [0, T]}\int_{\T}  \mez\rho &\big(u + \frac{\partial_x \rho}{\rho^2}\mu(\rho)\big)^2(x, t)\de x\\
&\le \left(\int_{\T}  \mez\rho \big(u + \frac{\partial_x \rho}{\rho^2}\mu(\rho)\big)^2(x, 0)\de x+C(1+\Vert f\Vert^2_{L^2(0, T; L^\infty)}\big)\big(1+\Vert \rho_0\Vert_{L^1(\T)}\big)\right)\exp(T)\\
&\le M(E_1, \Vert f\Vert_{L^2(0, T; L^\infty)}, T).
\end{align*}
Combined with \eqref{enCont1}, this implies the bound \eqref{entinqe1} when $\gamma\in (0, 1)$. 
\\
Next, we recall from \eqref{enCond3} the bound for $\Vert \rho^\gamma\Vert_{L^1(\T)}$. By the assumption that $\gamma\geq \alpha-\frac{1}{2}$, we obtain
 \[
 \Vert \rho^{\alpha - \frac{1}{2}}\Vert_{L^\infty(0,T;L^1)}\le C(1+ \Vert \rho^\gamma\Vert_{L^\infty(0,T;L^1)}{+ \Vert \rho\Vert_{L^\infty(0,T;L^1)}})\le  M(E_0, \Vert f\Vert_{L^2(0, T; L^\infty)}, T).
 \]
This combined with \eqref{entinqe1} and Nash's inequality
$$
\| \rho^{\alpha-\frac{1}{2}}\|_{L^\infty(0,T; L^2)} \leq C \| \rho^{\alpha-\frac{1}{2}}\|_{L^\infty(0,T; L^1)}^{2/3} \| \partial_x (\rho^{\alpha-\frac{1}{2}})\|_{L^\infty(0,T; L^2)}^{1/3}+ C \Vert\rho^{\alpha - \frac 12}\Vert_{L^\infty(0, T; L^1)}
$$
 leads to
\[
\|\rho^{\alpha-\frac{1}{2}} \|_{L^\infty(0,T; H^1)} \leq M(E_1, \Vert f\Vert_{L^2(0, T; L^\infty)}, T).
\]
The stated bound \eqref{rhoBnd1} then follows by Sobolev embedding $H^1\subseteq L^\infty$.

\vskip 0.5cm
2. In this case, $c_p <0$ and thus the  first term on the  right hand side of \eqref{entropyBalance} is positive and is equal to
\begin{align*}
 -\gamma c_p c_\mu\int_{\T} | \rho^{(\gamma + \alpha-3)/2}\partial_x \rho |^2 \de x &\leq   -2\gamma \frac{c_p }{c_\mu}\int_{\T}\rho^{\gamma -\alpha+1}\left(  | u+c_\mu\rho^{\alpha-2} \partial_x \rho |^2  + |u|^2\right)\de x\\
 &=-2\gamma \frac{c_p }{c_\mu}\int_{\T}\rho^{\gamma -\alpha} \left(  s(x,t)  -   \pi(\rho)  + \rho|u|^2\right)\de x.
 \end{align*}
Note that  \eqref{bound:rhsentropy} provides the bound
 \[
 \int_\T  f\rho \big(u + \frac{\partial_x \rho}{\rho^2}\mu(\rho)\big)\de x \leq \mez\Vert f(t)\Vert^2_{L^\infty(\T)}\Vert \rho_0\Vert_{L^1(\T)}+\int_\T  \mez\rho \big(u + \frac{\partial_x \rho}{\rho^2}\mu(\rho)\big)^2\de x.
 \]
In addition, since $\gamma\in (0, 1)$, part 2 of Lemma \ref{PiCont:0} provides a bound for $\pi(\rho)$ and $\rho u^2$. Moreover, note that when $c_p<0$ and $\gamma\in (0, 1)$ we have $\pi(\rho), s\ge 0$. Using these together with the assumption that $\gamma\le \alpha$ we have
\begin{align*}
	\frac{\de}{\de t}\int_{\T} s(x,t) \de x &\leq -2\gamma \frac{c_p }{c_\mu}\int_{\T}\rho^{\gamma -\alpha} \left(s(x,t)  -   \pi(\rho)  + \rho|u|^2\right)dx+\Vert f(t)\Vert^2_{L^\infty(\T)}\Vert \rho_0\Vert_{L^1(\T)}+\int_\T s(x, t)\de x.
\\ \nonumber
	&\leq -2\gamma \frac{c_p }{c_\mu}(\frac{1}{ \cl })^{\gamma -\alpha} \int_{\T} \left(s(x,t)  -   \pi(\rho)  + \rho|u|^2\right)\de x+\Vert f(t)\Vert^2_{L^\infty(\T)}\Vert \rho_0\Vert_{L^1(\T)}+\int_\T s(x, t)\de x.
\\ \nonumber
&\leq \Big(-2\gamma \frac{c_p }{c_\mu} (\frac{1}{ \cl })^{\gamma -\alpha} +1\Big)\int_{\T} s(x,t) \de x-2\gamma \frac{c_p }{c_\mu}(\frac{1}{ \cl })^{\gamma -\alpha} \int_{\T} \left(-\pi(\rho)  + \rho|u|^2\right)\de x\\
&\quad  +\Vert f(t)\Vert^2_{L^\infty(\T)}\Vert \rho_0\Vert_{L^1(\T)}\\
&\le \Big(-2\gamma \frac{c_p }{c_\mu} (\frac{1}{ \cl })^{\gamma -\alpha} +1\Big)\int_\T s(x, t)\de x+M(E_0, \Vert f\Vert_{L^2(0, T; L^\infty)}, \frac{1}{ \cl }, T)\\
&\quad+\Vert f(t)\Vert^2_{L^\infty(\T)}\Vert \rho_0\Vert_{L^1(\T)}
\end{align*}
for $t\le T$.  By Gr\"{o}nwall's lemma and \eqref{bound:s0}, we deduce that
\begin{align*}
\Vert s \Vert_{L^\infty(0, T; L^1)} &\le M(E_0+\Vert s(\cdot, 0)\Vert_{L^1(\T)}, \Vert f\Vert_{L^2(0, T; L^\infty)},{\frac{1}{ \cl }}, T)\\
& \le M(E_1, \Vert f\Vert_{L^2(0, T; L^\infty)}, \frac{1}{ \cl }, T).
\end{align*}
Combining this {with \eqref{enCont1}} gives
\be\label{otherrhobnd2}
\|\partial_x (\rho^{\alpha-\frac{1}{2}} ) \|_{L^\infty(0,T;L^2)} \leq M(E_1, \Vert f\Vert_{L^2(0, T; L^\infty)}, \frac{1}{ \cl }, T).
\ee
Since $\alpha-\frac{1}{2}\in(0, 1]$,  the mass conservation \eqref{enCond3} implies
\be
\Vert \rho^{\alpha - \frac{1}{2}}\Vert_{L^\infty(0,T;L^1)}\le C(1+\Vert \rho_0\Vert_{L^1(\T)}).
\ee
Combined with \eqref{otherrhobnd2}, this yields
\[
\Vert \rho^{\alpha - \frac{1}{2}} \Vert_{L^\infty(0,T; H^1)}\le M(E_1, \Vert f\Vert_{L^2(0, T; L^\infty)}, \frac{1}{ \cl }, T)
\]
from which \eqref{rhoBnd2} follows.
\\

\vspace{-4mm}
3. The bound \eqref{rhoBnd3} follows from \eqref{rhoBnd1} \& \eqref{entinqe1} and \eqref{rhoBnd2} \& \eqref{otherrhobnd2} respectively.

4. This follows from Propositions 4.5 and 4.6 in \cite{MV06}.
\end{proof}

\section{The active potential} \label{sec:goodUnknown}

We introduce in this section the \emph{active potential} $w:=- p(\rho) +\mu(\rho) \partial_xu$. This is a good unknown upon which much of the analysis is based. We first show that $w$ satisfies a {\it  forced quadratic heat equation with linear drift}.
\begin{prop}[$w$--equation] \label{LemwEqn} Let 
\be\label{wDef}
w:=- p(\rho) +\mu(\rho) \partial_xu.
\ee
Then $w$ satisfies
\begin{align} \nonumber
\partial_t w&=  \rho^{-1}\mu(\rho) \partial_x^2 w- (u+\mu(\rho)\frac{\partial_x \rho}{\rho^2} )\partial_xw+   \left(\rho \frac{p'(\rho)}{\mu(\rho)}-2 \frac{(\rho\mu'(\rho)+\mu(\rho))}{\mu(\rho)^2} p(\rho)\right)w\\
&\quad -\frac{(\rho\mu'(\rho)+\mu(\rho))}{\mu(\rho)^2}  w^2+ \left(  \rho\frac{p'(\rho)}{\mu(\rho)}-\frac{(\rho\mu'(\rho)+\mu(\rho))}{\mu(\rho)^2}  p(\rho)\right)p(\rho)+\mu(\rho)\p_xf.  \label{Weqn}
\end{align}
Moreover, the following balance holds
\bq\label{dt:w:L2}
\begin{aligned}
	\frac{\de}{\de t}\int_{\T}  \frac{1}{2} |w|^2(x,t)\de x  &= {-\int_{\T}  \rho^{-1}\mu(\rho) |\partial_x w|^2 dx- \int_{\T}\left(u+\frac{\mu'(\rho)}{\rho} \partial_x \rho \right) w \partial_x wdx }\\ 
	&\qquad  +\int_{\T}  \left(\rho \frac{p'(\rho)}{\mu(\rho)}-2\frac{(\rho\mu'(\rho)+\mu(\rho))}{\mu(\rho)^2} p(\rho)\right)|w|^2 dx-\int_{\T} \frac{(\rho\mu'(\rho)+\mu(\rho))}{\mu(\rho)^2}  w^3 d x\\
	&\qquad + \int_{\T} \left(\rho \frac{p'(\rho)}{\mu(\rho)}-\frac{(\rho\mu'(\rho)+\mu(\rho))}{\mu(\rho)^2} p(\rho)\right)p(\rho) wdx+\int_\T \mu(\rho)\p_xfwdx.
\end{aligned}
\eq
\end{prop}
\begin{proof}
From the definition of $w:=- p(\rho) +\mu(\rho) \partial_xu$ given by \eqref{wDef}, we compute
\begin{align}
\partial_x w &= (\partial_x \rho)  (-p'(\rho) +\mu'(\rho) \partial_x u) + \mu(\rho) \partial_x^2 u.
\end{align}
Thus, we have
\begin{align}\nonumber
\partial_t w &= (\partial_t \rho)  (-p'(\rho) +\mu'(\rho) \partial_x u) + \mu(\rho) \partial_t \partial_x u\\ \nonumber
&= - \partial_x (u \rho )  (-p'(\rho) +\mu'(\rho) \partial_x u) +\mu(\rho) \partial_t \partial_x u\\
&= -\rho \partial_x u   (-p'(\rho) +\mu'(\rho) \partial_x u) -  u(\partial_xw- \mu(\rho) \partial_x^2 u)+ \mu(\rho) \partial_t \partial_x u.
\end{align}
The momentum equation \eqref{eq:mom} gives
\begin{align*}
\partial_t u &= - u\partial_x u + \rho^{-1} \partial_x w {+f},\\
\partial_t \partial_x u &=-\partial_xu\partial_x u- u\partial_x^2u  - \frac{\partial_x \rho}{\rho^2} \partial_x w+ \rho^{-1} \partial_x^2 w+\p_xf.
\end{align*}
Combining the above results, we find
\begin{align*}
\partial_t w&= -\rho \partial_x u   (-p'(\rho) +\mu'(\rho) \partial_x u) -  u\partial_xw + u \mu(\rho) \partial_x^2 u\\
&\quad  -\mu(\rho)(|\partial_xu|^2+ u\partial_x^2u)  - \mu(\rho)\frac{\partial_x \rho}{\rho^2} \partial_x w+ \rho^{-1}\mu(\rho) \partial_x^2 w+\mu(\rho)\p_xf\\
&=  \rho^{-1}\mu(\rho) \partial_x^2 w+  \rho (\partial_x u) p'(\rho) -(\rho\mu'(\rho)+\mu(\rho))   |\partial_x u|^2 -  (u+\mu(\rho)\frac{\partial_x \rho}{\rho^2} )\partial_xw+\mu(\rho)\p_xf\\
&=  \rho^{-1}\mu(\rho) \partial_x^2 w+  \rho (w+ p(\rho) ) \frac{p'(\rho)}{\mu(\rho)}-\frac{(\rho\mu'(\rho)+\mu(\rho))}{\mu(\rho)^2}    (w+ p(\rho))^2 -  (u+\mu(\rho)\frac{\partial_x \rho}{\rho^2} )\partial_xw+\mu(\rho)\p_xf
\end{align*}
which, after rearrangement, establishes Eq. \eqref{Weqn}.    For the energy, multiplying the equation \eqref{Weqn} by $w$ yields
\begin{align*}
\partial_t\left(\frac{1}{2} |w|^2\right)&={ \partial_x\big(\frac{\mu(\rho)}{\rho} w \partial_x w\big) - \frac{\mu(\rho)}{\rho} |\partial_x w|^2 -\p_x\big(\frac{\mu(\rho)}{\rho}\big)w\p_xw  -\left(u+\frac{\mu(\rho)}{\rho^2} \partial_x \rho \right) w \partial_x w }  \\
&\quad + \left(\rho \frac{p'(\rho)}{\mu(\rho)}-2\frac{(\rho\mu'(\rho)+\mu(\rho))}{\mu(\rho)^2} p(\rho)\right)|w|^2 -\frac{(\rho\mu'(\rho)+\mu(\rho))}{\mu(\rho)^2}  w^3 \\
&\quad + \left(\rho \frac{p'(\rho)}{\mu(\rho)}-\frac{(\rho\mu'(\rho)+\mu(\rho))}{\mu(\rho)^2} p(\rho)\right)p(\rho) w+ \mu(\rho)\p_xfw.
\end{align*}
Integrating in space yields the balance.
\end{proof}
Let us remark that in \eqref{Weqn}  the new viscosity coefficient is $\frac{\mu(\rho)}{\rho}$ which is less degenerate than the original viscosity $\mu(\rho)$ for the momentum equation. In particular, when $\mu(\rho)=c_\mu\rho^\alpha$ with $\alpha\le 1$, $\frac{\mu(\rho)}{\rho}$ is not degenerate when $\rho$ goes to $0$. Energy estimates for the coupled system of $\rho$ and $w$ will allow us to control all the high Sobolev regularity of $\rho$ and $w$ as long as $\rho$ is positive. This leads to the proof of our continuation criterion in Theorem \ref{theo:cont}: no singularity occurs before vacuum formation.

 Furthermore, \eqref{Weqn} can be regarded as a nonlinear heat equation with variable coefficients. Note that the zero-order term in \eqref{Weqn} has the form $\lambda \rho^{2\gamma-\alpha}$ where $\lambda$ depends only on $c_\mu$ and $c_p$.  It can be readily seen that when the zero-order term and the forcing term in \eqref{Weqn} are nonpositive,  $w$ remains nonpositive if it is nonpositive initially. This fact will be exploited as the key ingredient in proving the existence of global solutions in Theorem \ref{theo:global2} when the viscosity is strongly degenerate.
 

\section{Proof of Theorem \ref{theo:cont}}\label{sec:4}
Throughout this section,  we suppose that
\be\label{AssumpRho}
0< \cl  \leq \rho(x,t)  \qquad t\in [0,T^*),\quad x\in \T.
\ee
and assume any of the following three conditions\\

\vspace{-5mm}
(i)   $c_p >0$ and $\alpha> \frac{1}{2}$, $\gamma\geq \alpha-\frac{1}{2} $, $\gamma\neq 1$ \\
(ii)  $c_p <0$ and  $\alpha\in (\frac{1}{2}, \frac{3}{2}]$, $0<\gamma\leq \alpha$, $\gamma<1$  \\
\red{(iii) $c_p>0$ and $\alpha\ge 0$, $\gamma>1$.}

Under these assumptions, by Lemma \ref{PiCont:1}, we have
\be\label{rhobd1}
\|\rho\|_{L^\infty(0,T;L^\infty(\T))}  \leq M(E_1, \Vert f\Vert_{L^2(0, T; L^\infty)}, \frac{1}{ \cl }, T),
\ee
and
\be\label{rhobd2}
\|\partial_x \rho\|_{L^\infty(0,T;L^2(\T))}  \leq M(E_1, \Vert f\Vert_{L^2(0, T; L^\infty)}, \frac{1}{ \cl }, T).
\ee
\begin{lemm}\label{lemm:wL2}
	\bq\label{w:L2}
	\begin{aligned}
		&\Vert w\Vert_{L^\infty(0, T; L^2)}+\Vert \partial_x w\Vert_{L^2(0, T; L^2)}+\Vert \partial_x u \Vert_{L^\infty(0, T; L^2)}+\Vert \partial_x^2  u \Vert_{L^2(0, T; L^2)}\\
		& \hspace{70mm} \le M(E_2, \Vert f\Vert_{L^2(0, T; H^1)}, \frac{1}{ \cl }, T),
	\end{aligned}
	\eq
	where $E_2=E_1+ \Vert \partial_x u_0\Vert_{L^2}$.
\end{lemm}
\begin{proof}
	As a consequence of {\eqref{AssumpRho}, \eqref{rhobd1}, and \eqref{dt:w:L2}}, there exist $c := c(E_1, \Vert f\Vert_{L^2(0, T; L^\infty)}, \frac{1}{ \cl }, T)>0$ and $C:=C(E_1, \Vert f\Vert_{L^2(0, T; L^\infty)}, \frac{1}{ \cl }, T)>0$ such that
	\begin{align} \nonumber
	\frac{\de}{\de t}\int_{\T}  \frac{1}{2} |w|^2(x,t)\de x  &\leq  - \frac{1}{c}\int_{\T}  |\partial_x w|^2 dx + \int_{\T}\left(|u|+C|\p_x\rho| \right) |w \partial_x w |dx \\
	&\qquad  +C\left(\int_{\T}  |w|^2 dx +  \int_{\T}  |w|^3 d x +\int_{\T}  |\p_x f|^2 d x +1 \right).
	\end{align}
	We bound 
	\[
	\int_\T \la\p_x wwu\ra\de x \le \Vert \p_x w\Vert_{L^2}\Vert w\Vert_{L^2}\Vert u\Vert_{L^\infty}\le C_1\Vert \p_x w\Vert_{L^2}\Vert w\Vert_{L^2}\Vert u\Vert_{H^1}\le \frac{1}{4c}\Vert \p_x w\Vert_{L^2}^2+C\Vert w\Vert^2_{L^2}\Vert u\Vert^2_{H^1}
	\]
	where $C_1$ denotes absolute constants throughout this proof. Next, applying Gagliardo-Nirenberg's inequality and Young's inequality implies
	\[
	\int_\T \la w\ra^3\de x \le \Vert w\Vert_{L^3}^3\le C_1(\Vert \p_x w\Vert_{L^2}^{\frac{1}{2}}\Vert w\Vert_{L^2}^{\frac 52}+\Vert w\Vert^3_{L^2})\le \frac{1}{4c}\Vert \p_x w\Vert_{L^2}^2+C\Vert w\Vert_{L^2}^{\frac{10}{3}}+C\Vert w\Vert^3_{L^2}
	\]
	and
	\begin{align*}
	\int_\T \la{\p_x ww\p_x \rho}\ra\de x &\le   \Vert \p_x w\Vert_{L^2}\Vert w\Vert_{L^\infty}\Vert \p_x \rho\Vert_{L^2}\\
	&\le C_1 \Vert \p_x w\Vert_{L^2}(\Vert \p_x w\Vert_{L^2}^\mez\Vert w\Vert_{L^2}^\mez+\Vert w\Vert_{L^2})\Vert \p_x \rho\Vert_{L^2}\\
	&\le   C_1\Vert \p_x w\Vert_{L^2}^{\tdm}\Vert w\Vert_{L^2}^\mez\Vert \p_x \rho\Vert_{L^2}+ C_1\Vert \p_x w\Vert_{L^2}\Vert w\Vert_{L^2}\Vert \p_x \rho\Vert_{L^2}\\
	&\le \frac{1}{4c}\Vert \p_x w\Vert^2_{L^2}+C\Vert w\Vert^2_{L^2}\Vert \p_x \rho\Vert^4_{L^2}+C\Vert w\Vert_{L^2}^2\Vert \p_x \rho\Vert_{L^2}^2.
	\end{align*}
	
	Putting together the above bounds, and interpolating, yields the following inequality
	\begin{equation}\label{eq:togron}
	\begin{aligned}
	\frac 12 \frac{\de}{\de t} \Vert w\Vert^2_{L^2} + \frac{1}{4c} \Vert \p_x w\Vert^2_{L^2}\leq C \Vert{w}\Vert^2_{L^2} (\Vert w \Vert^2_{L^2} + \Vert \p_x \rho \Vert_{L^2}^4+ 1) + C\Vert \p_x f \Vert_{L^2}^2 + C.
	\end{aligned}
	\end{equation}
	In view of~\eqref{rhobd2}, we have
	\begin{equation*}
	\int_0^{T} \Vert \p_x \rho (\cdot, t)\Vert^4_{L^2} \de t \leq M(E_1, \Vert f\Vert_{L^2(0, T; L^\infty)}, \frac{1}{ \cl }, T).
	\end{equation*}
	Furthermore, using the definition of $w$ together with bounds~\eqref{rhobd1} \&~\eqref{enCont2}, we have
	\begin{equation*}
	\Vert w\Vert_{L^2(0,T; L^2)} \leq M(E_1, \Vert f\Vert_{L^2(0, T; L^\infty)}, \frac{1}{ \cl }, T).
	\end{equation*}
	The last two displays, together with Gr\"onwall's lemma applied to~\eqref{eq:togron}, yields the bound
	\begin{equation*}
	\begin{aligned}
	&\Vert w\Vert_{L^\infty(0, T; L^2(\T))}+\Vert \p_x w\Vert_{L^2(0, T; L^2(\T))}\\
	&\qquad \qquad \le M(\Vert w_0\Vert_{L^2}, c, C, E_1, \Vert f\Vert_{L^1(0,T; H^1)}, \frac 1 { \cl }, T)\le M(E_1, \Vert f\Vert_{L^1(0,T; H^1)}, \frac 1 { \cl }, T).
	\end{aligned}
	\end{equation*}
	Here, we used the fact that
	\begin{equation*}
	\Vert w_0 \Vert^2_{L^2} \leq 2 c_p^2\Vert \rho_0 \Vert_{L^\infty}^{2\gamma} +2 c_\mu^2\Vert \rho_0\Vert_{L^\infty}^{2\alpha}\Vert \p_x u_0\Vert^2_{L^2}.
	\end{equation*}
	
	The above bound can be used to obtain similar estimates for $\Vert \partial_x u \Vert_{L^\infty(0, T; L^2)}$ and $\Vert \partial_x^2  u \Vert_{L^2(0, T; L^2)}$ directly from the definition of $w$ \eqref{wDef}.
\end{proof}

\begin{lemm}\label{lemm:wH1}
	\bq\label{w:H1}
	\begin{aligned}
		\Vert \partial_x^2\rho\Vert_{L^\infty(0, T;  L^2)}&+\Vert \partial_xw\Vert_{L^\infty(0, T; L^2)}+\Vert \partial_x^2w\Vert_{L^2(0, T; L^2)}\\
		&\quad +\Vert \partial_x^2u\Vert_{L^\infty(0, T; L^2)}+\Vert \partial_x^3u\Vert_{L^2(0, T; L^2)}\le M(E_3, \Vert f\Vert_{L^1(0,T; H^1)}, \frac 1 { \cl }, T)
	\end{aligned}
	\eq
	where
	\[
	E_3=E_2+\Vert \partial_x^2 \rho_0\Vert_{L^2}+\Vert \partial_x^2u_0\Vert_{L^2}.
	\]
\end{lemm}

\begin{proof}   To prove this lemma, we obtain energy estimates for the mass equation \eqref{eq:mass} and the $w$--equation \eqref{Weqn} simultaneously. The proof proceeds in 4 steps.
	
	{\bf Step 1.} Let $m\ge 2$ be an arbitrary integer. Differentiating equation \eqref{eq:mass} $m$ times, then multiplying the resulting equation by $\p_x^m\rho$ and integrating in space we get
	\begin{align*}
	\mez\frac{\de}{\de t}\int_\T |\p_x^m\rho|^2&=-\int_{\T} \p_x^m(u\p_x \rho)\p_x^m\rho-\int_\T \p_x^m(\rho \p_x u)\p_x^m\rho\\
	&=-\int_\T u\p_x\p_x^m\rho \p_x^m\rho-\int_\T \big([\p_x^m, u]\p_x \rho\big)\p_x^m\rho-\int_\T \big([\p_x^m, \rho]\p_x u \big)\p_x^m\rho-\int_\T \rho\p_x^{m+1} u \p_x^m\rho.
	\end{align*}
	Using the Kato-Ponce commutator estimate \cite{KP} and the inequality 
	\[
	\Vert \p_x g \Vert_{L^\infty(\T)}\le C\Vert \p_x^2 g \Vert_{L^2(\T)}\le C_n\Vert \p_x^n g\Vert_{L^2(\T)}\quad\forall n\ge 3,
	\]
	we have
	\[
	\Vert [\p_x^m, u]\p_x \rho\Vert_{L^2}\le C\Vert \p_x u\Vert_{L^\infty}\Vert \p_x^{m-1}\p_x \rho\Vert_{L^2}+C\Vert \p_x^mu\Vert_{L^2}\Vert \p_x \rho\Vert_{L^\infty}\le C\Vert \p_x^mu\Vert_{L^2}\Vert \p_x^m\rho\Vert_{L^2}
	\]
	and
	\[
	\Vert [\p_x^m, \rho]\p_x u\Vert_{L^2}\le C\Vert \p_x \rho\Vert_{L^\infty}\Vert \p_x^{m-1}\p_x u\Vert_{L^2}+C\Vert \p_x^m\rho\Vert_{L^2}\Vert \p_x u\Vert_{L^\infty}\le C\Vert  \p_x^mu\Vert_{L^2}\Vert \p_x^m\rho\Vert_{L^2}.
	\]
	In addition,
	\[
	\la \int_\T u\p_x\p_x^m\rho\p_x^m\rho\ra =\frac 12 \la \int_\T \p_x u|\p_x^m\rho|^2\ra\le \frac 12 \Vert \p_x u\Vert_{L^\infty}\Vert \p_x^m\rho\Vert_{L^2}^2 \le C\Vert  \p_x^mu\Vert_{L^2}\Vert \p_x^m\rho\Vert^2_{L^2}.
	\]
	We thus obtain
	\bq\label{dt:dmh:L2}
	\frac{\de}{\de t}\Vert \p_x^m\rho\Vert_{L^2}^2\le C\Vert  \p_x^mu\Vert_{L^2}\Vert \p_x^m\rho\Vert_{L^2}^2+\Vert \rho\Vert_{L^\infty}\Vert \p_x^{m+1}u\Vert_{L^2}\Vert \p_x^m\rho\Vert_{L^2}.
	\eq
	{\bf Step 2.} Recall equation \eqref{Weqn} with power-law pressure and viscosity
	\bq\label{eq:w2}
	\begin{aligned}
	\partial_t w&= c_\mu \rho^{\alpha-1} \p_x^2 w- (u+c_\mu \rho^{\alpha-2} \p_x \rho )\p_x w
	+    \frac{c_p}{c_\mu} \left(\gamma  -2(\alpha  +1) \right)\rho^{\gamma-\alpha} w\\
	&\quad - \frac{1}{c_\mu}(\alpha  +1)\rho^{-\alpha} w^2+  \frac{c^2_p}{c_\mu} \left(\gamma  -(\alpha  +1) \right)\rho^{2 \gamma-\alpha}+c_\mu \rho^\alpha \p_x f.
	\end{aligned}
	\eq
	Differentiating in space, multiplying the resulting equation by $\p_x w$ and integrating by parts in $x$ leads to
	\begin{align*}
	\mez\frac{\de}{\de t}\int_\T |\p_x w|^2&=- c_\mu\int_\T \rho^{\alpha-1} |\p_x^2 w|^2 + \int_\T(u+c_\mu \rho^{\alpha-2} \p_x \rho )\p_x w\p_x^2 w  +  \frac{c_p}{c_\mu} \left(\gamma  -2(\alpha  +1) \right)\int_\T |\p_x w|^2\rho^{\gamma-\alpha}\\
	&\quad + \frac{c_p}{c_\mu}(\gamma-\alpha) \left(\gamma  -2(\alpha  +1) \right)\int_\T w \rho^{\gamma-\alpha-1} \p_x w\p_x \rho\\
	&\quad -\frac{2}{c_\mu}(\alpha  +1)\int_\T \rho^{-\alpha}w|\p_x w|^2 +\frac{\alpha }{c_\mu}(\alpha  +1)\int_\T w^2\p_x w\p_x \rho \rho^{-\alpha-1}\\
	&\quad  + \frac{c^2_p}{c_\mu}(2\gamma-\alpha) \left(\gamma  -(\alpha  +1) \right) \int_\T \rho^{2\gamma-\alpha-1} \p_x w\p_x \rho -c_\mu \int_{\T} \rho^\alpha \p_x^2 w \p_x f\\
	&=:- c_\mu\int_\T \rho^{\alpha-1} |\p^2_x w|^2+\sum_{j=1}^7 H_j.
	\end{align*}
	after integrating by parts. By virtue of \eqref{AssumpRho} and \eqref{rhobd1}, there exists $c:=c(E_1, \Vert f\Vert_{L^2(0, T; L^\infty)}, \frac{1}{ \cl }, T)>0$ such that
	\[
	c_\mu\int_\T \rho^{\alpha-1} |\p_x^2 w|^2\ge \frac{1}{c}\int_\T |\p_x^2 w|^2.
	\]
	Note, under our assumptions $\rho$ and $1/\rho$ are bounded (see \eqref{AssumpRho} and \eqref{rhobd1}).  Therefore all coefficients involving $L^\infty$ norms of $\rho$ to some power can be bounded by some constant $C= M(E_1, \Vert f\Vert_{L^2(0, T; L^\infty)}, \frac{1}{ \cl }, T, \gamma, \alpha)$.  The constant may change line by line.  
	\begin{itemize}
		\item Estimate for $H_1$:
		\begin{align*}
		\la \int_\T  (u+c_\mu \rho^{\alpha-2} \p_x \rho )\p_x w\p^2_x w\ra &\le \Vert \p^2_x w\Vert_{L^2}\Vert \p_x w\Vert_{L^2}\Vert u\Vert_{L^\infty}+C\Vert \p^2_x w\Vert_{L^2}\Vert \p_x w\Vert_{L^2}\Vert \p_x \rho\Vert_{L^\infty}\\
		&\le \frac{1}{10c}\Vert \p^2_x w\Vert^2_{L^2}+C\Vert \p_x w\Vert^2_{L^2}\Vert u\Vert^2_{H^1}+C\Vert \p_x w\Vert_{L^2}^2\Vert \p^2_x \rho\Vert^2_{L^2}.
		\end{align*}
		\item Estimate for $H_2$:
		\[
		\la \int_\T |\p_x w|^2\rho^{\gamma-\alpha} \ra \le C\Vert \p_x w\Vert_{L^2}^2.
		\]
		\item Estimate for $H_3$:
		\begin{align*}
		\la  \int_\T w\p_x w\p_x \rho \rho^{\gamma-\alpha-1} \ra &\le \| \rho^{\gamma-\alpha-1}\|_\infty  \Vert w\Vert_{L^\infty}\Vert \p_x w\Vert_{L^2}\Vert \p_x \rho\Vert_{L^2}\\
		&\le C\Vert w\Vert_{L^2}\Vert \p_x w\Vert_{L^2}\Vert \p_x \rho\Vert_{L^2}+C\Vert \p_x w\Vert_{L^2}^2\Vert \p_x \rho\Vert_{L^2}.
		\end{align*}
		\item Estimate for $H_4$:
		\begin{align*}
		\la\int_\T \rho^{-\alpha}w|\p_x w|^2 \ra &\le \frac{1}{ \cl ^\alpha}\Vert w\Vert_{L^\infty}\Vert \p_x w\Vert_{L^2}^2\le \frac{1}{4 \cl ^{2\alpha}}\Vert w\Vert_{L^\infty}^2+C\Vert \p_x w\Vert_{L^2}^4\\
		&\le C\Vert w\Vert_{H^1}^2+C\Vert \p_x w\Vert_{L^2}^4 .
		\end{align*}
		\item Estimate for $H_5$:
		\begin{align*}
		\la \int_\T w^2\p_x w\p_x \rho \rho^{-\alpha-1} \ra &\le \frac{1}{ \cl ^{1+\alpha}}\Vert \p_x w\Vert_{L^2}\Vert w\Vert_{L^\infty}^2\Vert \p_x \rho\Vert_{L^2}\\
		&\le C \Vert \p_x w\Vert_{L^2}\Vert w\Vert_{H^1}^2\Vert \p_x \rho\Vert_{L^2}\\
		&\le C \Vert \p_x w\Vert_{L^2}\Vert w\Vert_{L^2}^2\Vert \p_x \rho\Vert_{L^2}+ C \Vert \p_x w\Vert_{L^2}^3\Vert \p_x \rho\Vert_{L^2}.
		\end{align*}
		\item Estimate for $H_6$:
		\[
		\la  \int_\T \rho^{\gamma-\alpha-1} \p_x w\p_x \rho \ra \le  C\Vert \p_x w\Vert_{L^2}\Vert \p_x \rho\Vert_{L^2}.
		\]
		\item Estimate for $H_7$:
		\[
		\la   \int_{\T} \rho^\alpha \p_x^2 w \p_x f \ra \le  \frac 1 {10 c} \Vert \p^2_x w\Vert^2_{L^2} + C \Vert \p_x f\Vert^2_{L^2}.
		\]
	\end{itemize}
	Putting together the above estimates gives
	\bq\label{dt:dw:L2}
	\begin{aligned}
		&\frac{\de}{\de t}\Vert \p_x w\Vert_{L^2}^2+\frac{1}{2c}\Vert \p^2_x w\Vert_{L^2}^2\\
		&\le C  \left( \Vert \p_x w\Vert^2_{L^2}\Vert u\Vert^2_{H^1}+ \Vert \p_x w\Vert_{L^2}^2\Vert \p^2_x \rho\Vert^2_{L^2}+\Vert \p_x w\Vert_{L^2}^4 + \Vert \p_x w\Vert_{L^2}^3\Vert \p_x \rho\Vert_{L^2}\right)+G
	\end{aligned}
	\eq
	with 
	\begin{align*}
	G&= C  \left( \Vert \rho\Vert_{L^\infty}\Vert \p_x w\Vert_{L^2}^2+\Vert w\Vert_{L^2}\Vert \p_x w\Vert_{L^2}\Vert \p_x \rho\Vert_{L^2}+\Vert \p_x w\Vert_{L^2}^2\Vert \p_x \rho\Vert_{L^2}\right.\\
	&\left.\qquad  +\Vert w\Vert_{H^1}^2+\Vert \p_x w\Vert_{L^2}\Vert w\Vert_{L^2}^2\Vert \p_x \rho\Vert_{L^2}+ \Vert \p_x w\Vert_{L^2}\Vert \p_x \rho\Vert_{L^2}+\Vert \p_x f\Vert^2_{L^2}\right).
	\end{align*}
	By virtue of the estimates \eqref{rhobd1}, \eqref{rhobd2} and \eqref{w:L2} we deduce that 
	\[
	\Vert G\Vert_{L^1((0, T))}\le M(E_2, \Vert f\Vert_{L^2(0, T; H^1)}, \frac{1}{ \cl }, T).
	\]
	{\bf Step 3.} Letting $m=2$ in \eqref{dt:dmh:L2} and using the embedding $H^1(\T)\subset L^\infty(\T)$ we get
	\[
	\frac{\de}{\de t}\Vert \p^2_x \rho\Vert_{L^2}^2\le C\Vert \p^2_x u\Vert_{L^2}\Vert \p^2_x \rho\Vert_{L^2}^2+C\Vert \rho\Vert_{H^1}\Vert \p^3_x u\Vert_{L^2}\Vert \p^2_x \rho\Vert_{L^2}.
	\]
	Recalling the definition \eqref{wDef} $w=- c_p\rho^\gamma +c_\mu \rho^\alpha \partial_xu$ we have
	\begin{align}\nonumber
	\p^3_x u&=\p_x^2(\frac{w}{c_\mu \rho^\alpha}+\frac{c_p}{c_\mu} \rho^{\gamma-\alpha} )\\ \nonumber
	&=\frac{\p^2_x w}{c_\mu \rho^\alpha}-2\alpha \frac{\p_x w\p_x \rho}{c_\mu \rho^{\alpha+1}}-\alpha \frac{w\p^2_x \rho}{c_\mu \rho^{\alpha+1}}+\alpha(\alpha+1)\frac{w|\p_x \rho|^2}{c_\mu \rho^{\alpha+2}}\\ \label{utriple}
	&\quad +\frac{c_p}{c_\mu}  (\gamma-\alpha)\p^2_x \rho \rho^{\gamma-\alpha-1} + \frac{c_p}{c_\mu}  (\gamma-\alpha)(\gamma-\alpha-1)|\p_x \rho|^2 \rho^{\gamma-\alpha-2}.
	\end{align}
	Consequently
	\begin{align*}
	\Vert \p^3_x u\Vert_{L^2}& \leq C  \left( \Vert \p^2_x w\Vert_{L^2}+\Vert \p_x w\Vert_{L^2}\Vert \p_x \rho\Vert_{L^\infty}+\Vert w\Vert_{H^1}\Vert \p^2_x \rho\Vert_{L^2}\right.\\
	&\left. \qquad +\Vert w\Vert_{L^\infty}\Vert \p_x \rho\Vert_{L^2}\Vert \p_x \rho\Vert_{L^\infty}+ \| \rho^{\gamma-\alpha-1}\|_\infty \Vert \p^2_x \rho\Vert_{L^2}+  \| \rho^{\gamma-\alpha-2}\|_\infty \Vert \p_x \rho\Vert_{L^2}\Vert \p_x \rho\Vert_{L^\infty}\right).
	\end{align*}
	Therefore, we obtain
	\bq\label{dt:d2h:L2}
	\begin{aligned}
		&\frac{\de}{\de t}\Vert \p^2_x \rho\Vert_{L^2}^2\\
		&\le C  \left( \Vert \p^2_x u\Vert_{L^2}\Vert \p^2_x \rho\Vert_{L^2}^2+ \Vert \rho\Vert_{H^1}\Vert \p^2_x w\Vert_{L^2}\Vert \p^2_x \rho\Vert_{L^2}+\Vert \rho\Vert_{H^1}\Vert \p_x w\Vert_{L^2}\Vert \p^2_x \rho\Vert_{L^2}\Vert \p_x \rho\Vert_{L^\infty}\right. \\
		&\qquad +\Vert w\Vert_{H^1}\Vert \rho\Vert_{H^1}\Vert \p^2_x \rho\Vert^2_{L^2} +\Vert w\Vert_{L^\infty}\Vert \rho\Vert_{H^1}^2\Vert \p_x \rho\Vert_{L^\infty}\Vert \p^2_x \rho\Vert_{L^2}\\
		&\qquad \left.+\Vert \rho\Vert_{H^1}\Vert \p^2_x \rho\Vert^2_{L^2} + \Vert \rho\Vert_{H^1}^2\Vert \p^2_x \rho\Vert_{L^2}^2\right)\\
		&\le \frac{1}{10c}\Vert \p^2_x w\Vert^2_{L^2} +C\left( \Vert \p^2_x u\Vert_{L^2}\Vert \p^2_x \rho\Vert_{L^2}^2+\Vert \rho\Vert^2_{H^1}\Vert \p^2_x \rho\Vert^2_{L^2}+\Vert \rho\Vert_{H^1}\Vert \p_x w\Vert_{L^2}\Vert \p^2_x \rho\Vert_{L^2}^2\right. \\
		&\qquad \left.+\Vert w\Vert_{H^1}\Vert \rho\Vert_{H^1}\Vert \p^2_x \rho\Vert^2_{L^2}+\Vert w\Vert_{H^1}\Vert \rho\Vert_{H^1}^2\Vert \p^2_x \rho\Vert^2_{L^2}+\Vert \rho\Vert_{H^1}\Vert \p^2_x \rho\Vert^2_{L^2}+\Vert \rho\Vert_{H^1}^2\Vert \p^2_x \rho\Vert_{L^2}^2\right) \\
		&\le \frac{1}{10c}\Vert \p^2_x w\Vert^2_{L^2}+F\Vert \p^2_x \rho\Vert^2_{L^2},
	\end{aligned}
	\eq
	with
	\[
	\begin{aligned}
	F=\quad &C  \left(\Vert \p^2_x u\Vert_{L^2}+\Vert \rho\Vert^2_{H^1}+\Vert \rho\Vert_{H^1}\Vert \p_x w\Vert_{L^2}\right.\\
	&\quad \left.+\Vert w\Vert_{H^1}\Vert \rho\Vert_{H^1}+\Vert w\Vert_{H^1}\Vert \rho\Vert_{H^1}^2+ \Vert \rho\Vert_{H^1}+ \Vert \rho\Vert_{H^1}^2\right).
	\end{aligned}
	\]
	Combining the estimates \eqref{rhobd1}, \eqref{rhobd2} and \eqref{w:L2} yields
	\[
	\Vert F\Vert_{L^1((0, T))}\le M(E_2, \Vert f \Vert_{L^2(0,T; H^1(\T))},\frac{1}{ \cl },T).
	\]
	{\bf Step 4.} Adding \eqref{dt:d2h:L2} to \eqref{dt:dw:L2} leads to
	\bq\label{dt:energy:2}
	\begin{aligned}
		\frac{\de}{\de t}(\Vert \p^2_x \rho\Vert_{L^2}^2+\Vert \p_x w\Vert_{L^2}^2)+\frac{1}{4c}\Vert \p^2_x w\Vert_{L^2}^2&\le \Vert \p_x w\Vert^2_{L^2}H+\Vert \p^2_x \rho\Vert^2_{L^2}(F+C\Vert \p_x w\Vert_{L^2}^2)+G\\
		&\le  (\Vert \p_x w\Vert^2_{L^2}+\Vert \p^2_x \rho\Vert_{L^2})(H+F+C \Vert \p_x w\Vert_{L^2}^2)+G
	\end{aligned}
	\eq
	with 
	\[
	H= C  \left(\Vert u\Vert^2_{H^1}+\Vert \p_x w\Vert_{L^2}^2+\Vert \p_x w\Vert_{L^2}\Vert \p_x \rho\Vert_{L^2}\right)
	\]
	satisfying, in virtue of~\eqref{rhobd1},~\eqref{rhobd2} and~\eqref{w:L2},
	\[
	\Vert H\Vert_{L^1((0, T))}\le  M(E_2, \Vert f\Vert_{L^2(0, T; H^1)}, \frac{1}{ \cl }, T).
	\]
	Finally, we integrate \eqref{dt:energy:2} in time, then apply Gr\"onwall's lemma, the estimates for $F$, $G$ and $H$, and the estimate \eqref{w:L2} on $\Vert \p_x w\Vert_{L^2(0, T; L^2)}$ to obtain
	\begin{align*}
	&\Vert \p^2_x \rho\Vert_{L^\infty(0, T;  L^2)}+\Vert \p_x w\Vert_{L^\infty(0, T; L^2)}+\frac{1}{c}\Vert \p^2_x w\Vert_{L^2(0, T; L^2)}\\
	&\qquad \qquad \le M(E_2, \Vert f\Vert_{L^2(0, T; H^1)}, \frac{1}{ \cl }, T, \Vert \p^2_x \rho_0 \Vert_{L^2},\Vert \p_x w_0\Vert_{L^2} )\\
	&\qquad \qquad \le M(E_3, \Vert f\Vert_{L^2(0, T; H^1)}, \frac{1}{ \cl }, T),
	\end{align*}
	where 
	\[
	E_3=E_2+\Vert \partial_x^2 \rho_0\Vert_{L^2}+\Vert \partial_x^2u_0\Vert_{L^2}.
	\]
	It then follows easily that 
	\[
	\Vert \p^2_x u\Vert_{L^\infty(0, T; L^2)}+\Vert \p^3_x u\Vert_{L^2(0, T; L^2)}\le  M(E_3, \Vert f\Vert_{L^2(0, T; H^1)}, \frac{1}{ \cl }, T).
	\]
\end{proof}

\begin{lemm}\label{lemm:wH2}
For any $k\ge 2$ there exists $M_k$ depending only on $k$ such that
	\bq\label{w:Hk-1}
	\begin{aligned}
		\Vert \partial_x^k\rho\Vert_{L^\infty(0, T;  L^2)}&+\Vert \partial_x^{k-1}w\Vert_{L^\infty(0, T; L^2)}+\Vert \partial_x^kw\Vert_{L^2(0, T; L^2)}\\
		&\quad +\Vert \partial_x^{k}u\Vert_{L^\infty(0, T; L^2)}+\Vert \partial_x^{k+1}u\Vert_{L^2(0, T; L^2)}\le M_k\big(E_{k+1}, \Vert f\Vert_{L^2(0, T; H^{k-1})}, \frac{1}{ \cl }, T\big)
	\end{aligned}
	\eq
	where
	\[
	E_{k+1}=E_{k}+\Vert \partial_x^k \rho_0\Vert_{L^2}+\Vert \partial_x^ku_0\Vert_{L^2}.
	\]
\end{lemm}

\begin{proof}
	The proof proceeds by induction in $k$. According to Lemma \ref{lemm:wH1}, \eqref{w:Hk-1} holds for $k=2$. Assuming that  \eqref{w:Hk-1} holds for $k-1$ with $k\ge 3$,  to obtain it for $k$ we perform $H^k$ energy estimate for $\rho$ and $H^{k-1}$ energy estimate for $w$. This follows along the same lines as that of Lemma \ref{lemm:wH1}. We first apply \eqref{dt:dmh:L2} with $m=k$  to have
	\bq\label{dt:d3h}
	\begin{aligned}
	\frac{\de}{\de t}\Vert \p^k_x \rho\Vert_{L^2}^2&\le C\Vert \p_x^k u\Vert_{L^2}\Vert \p^k_x \rho\Vert_{L^2}^2+\Vert \rho\Vert_{L^\infty}\Vert \p_x^{k+1} u\Vert_{L^2}\Vert \p^k_x \rho\Vert_{L^2}\\
	&\le M\big(E_{k}, \Vert f\Vert_{L^2(0, T; H^{k-2})}, \frac{1}{ \cl }, T\big)\Big(\Vert  \p_x^ku\Vert_{L^2}\Vert \p^k_x \rho\Vert_{L^2}^2+\Vert \p_x^{k+1} u\Vert_{L^2}\Vert \p^k_x \rho\Vert_{L^2}\Big).
	\end{aligned}
	\eq
	By differentiating  $k$ times the formula 
	\[
	\p_xu=\frac{1}{c_\mu}w\rho^{-\alpha}+c_p\rho^{\gamma-\alpha}
	\]
	and using the induction hypothesis together with the fact that $k\ge 3$ we obtain
	\begin{align*}
	\Vert \p_x^{k+1} u\Vert_{L^2}&\le C\Vert [\p_x^k, \rho^{-\alpha}]w\Vert_{L^2}+C\Vert  \rho^{-\alpha}\p_x^kw\Vert_{L^2}+\Vert \p_x^k\rho^{\gamma-\alpha}\Vert_{L^2}\\
	&\le C\Vert \p_x\rho^{-\alpha}\Vert_{L^\infty}\Vert w\Vert_{H^{k-1}}+C\Vert  \rho^{-\alpha}\Vert_{H^k}\Vert w\Vert_{L^\infty}+C\Vert  \rho^{-\alpha}\Vert_{L^\infty}\Vert\p_x^kw\Vert_{L^2}+\Vert \p_x^k\rho^{\gamma-\alpha}\Vert_{L^2}\\
	 &\le C\Vert \rho^{-\alpha}\Vert_{H^2}\Vert w\Vert_{H^{k-1}}+C\Vert  \rho^{-\alpha}\Vert_{H^k}\Vert w\Vert_{H^1}++C\Vert  \rho^{-\alpha}\Vert_{H^1}\Vert\p_x^kw\Vert_{L^2}+\Vert \rho^{\gamma-\alpha}\Vert_{H^k}\\
	 &\le M\big(E_{k}, \Vert f\Vert_{L^2(0, T; H^{k-2})}, \frac{1}{ \cl }, T\big)\big(\Vert \p_x^kw\Vert_{L^2}+\Vert \p_x^k\rho\Vert_{L^2}+1\big).
	\end{align*}
	It then follows from \eqref{dt:d3h} that
	\bq\label{dtrhoHk}
	\begin{aligned}
	\frac{\de}{\de t}\Vert \p^k_x \rho\Vert_{L^2}^2&\le M\big(E_{k}, \Vert f\Vert_{L^2(0, T; H^{k-2})}, \frac{1}{ \cl }, T\big)\Big[\Vert \p^k_x \rho\Vert_{L^2}^2\big(\Vert \p_x^{k} u\Vert_{L^2}+1\big)+\Vert \p_x^kw\Vert_{L^2}\Vert \p^k_x \rho\Vert_{L^2}+1\Big]\\
	&\le \frac{1}{10c}\Vert \p_x^kw\Vert_{L^2}^2+M\big(E_{k}, \Vert f\Vert_{L^2(0, T; H^{k-2})}, \frac{1}{ \cl }, T\big)\Big[\Vert \p^k_x \rho\Vert_{L^2}^2\big(\Vert \p_x^{k} u\Vert_{L^2}+1\big)+1\Big]
	\end{aligned}
	\eq
	where $c=c(E_1, \Vert f\Vert_{L^2(0, T; L^\infty)}, \frac{1}{ \cl }, T)>0$ be a positive number such that 
\[
\rho^{\alpha-1}\ge \frac{1}{c}\quad\forall (x, t)\in \T\times [0, T^*).
\]
	 Next, we differentiate equation \eqref{eq:w2} $k-1$ times  in $x$, multiply the resulting equation by $\p_x^{k-1}w$ and integrate over $\T$. We estimate successively each resulting term on the right hand side of \eqref{eq:w2}.\\
1. The dissipation term:
\begin{align*}
	  \int_\T\p_x^{k-1} &\big(\rho^{\alpha-1} \p_x^2 w\big)\p_x^{k-1}w=- \int_\T\p_x^{k-2}\big(\rho^{\alpha-1} \p_x^2 w\big)\p_x^kw\\
	 &= - \int_\T\rho^{\alpha-1}|\p_x^kw|^2- \int_\T \p_x^kw\sum_{\ell=1}^{k-2}C_\ell\p_x^{\ell} \rho^{\alpha-1}\p_x^{k-\ell}w\\
	 &\le -\frac{1}{c}\Vert\p_x^kw\Vert_{L^2}^2+C\Vert \p_x^kw\Vert_{L^2}\sum_{\ell=1}^{k-2}C_\ell\Vert \p_x^{\ell} \rho^{\alpha-1}\Vert_{L^\infty}\Vert \p_x^{k-\ell}w\Vert_{L^2}\\
	 &\le -\frac{1}{c}\Vert\p_x^kw\Vert_{L^2}^2+C\Vert \p_x^kw\Vert_{L^2}\Vert \rho\Vert_{H^{k-1}}\big(\Vert \p_x^{k-1}w\Vert_{L^2}+\Vert w\Vert_{L^2}\big)\\
	 &\le -\frac{1}{2c}\Vert\p_x^kw\Vert_{L^2}^2+C'\Vert \rho\Vert_{H^{k-1}}^2\big(\Vert \p_x^{k-1}w\Vert^2_{L^2}+\Vert w\Vert^2_{L^2}\big)\\
	 &\le -\frac{1}{2c}\Vert\p_x^kw\Vert_{L^2}^2+M\big(E_{k}, \Vert f\Vert_{L^2(0, T; H^{k-2})}, \frac{1}{ \cl }, T\big)\big(\Vert \p_x^{k-1}w\Vert^2_{L^2}+1\big).
\end{align*}
2. The drift term. We have
\begin{align*}
 & \int_\T\p_x^{k-1}\big(u\p_xw+c_\mu \rho^{\alpha-2} \p_x \rho\p_xw\big)\p_x^{k-1}w&=- \int_\T\p_x^{k-2}\big(u\p_xw\big)\p_x^kw-c_\mu \int_\T\p_x^{k-2}\big(\p_x \frac{\rho^{\alpha-1}}{\alpha-1}\p_xw\big)\p_x^kw
\end{align*}
where we adopted the convention $\frac{\rho^{\alpha-1}}{\alpha-1}=\ln \rho$ when $\alpha=1$. Noting that $H^{k-2}(\T)$ is an algebra for $k\ge 3$, we then bound
\begin{align*}
&\hspace{-10mm} \la \int_\T\p_x^{k-1}\big(u\p_xw+c_\mu \rho^{\alpha-2} \p_x \rho\p_xw\big)\p_x^{k-1}w\ra \\
&\le C\Vert \p_x^kw\Vert_{L^2}\Vert u\Vert_{H^{k-2}}\Vert w\Vert_{H^{k-1}}+C\Vert \p_x^kw\Vert_{L^2}\Vert \frac{\rho^{\alpha-1}}{\alpha-1}\Vert_{H^{k-1}}\Vert w\Vert_{H^{k-1}}\\
&\le \frac{1}{20c}\Vert \p_x^kw\Vert_{L^2}^2+C'\Vert u\Vert_{H^{k-2}}^2\Vert w\Vert_{H^{k-1}}^2+C'\Vert \frac{\rho^{\alpha-1}}{\alpha-1}\Vert_{H^{k-1}}^2\Vert w\Vert_{H^{k-1}}^2\\
&\le  \frac{1}{20c}\Vert \p_x^kw\Vert_{L^2}^2+M\big(E_{k}, \Vert f\Vert_{L^2(0, T; H^{k-2})}, \frac{1}{ \cl }, T\big)\big(\Vert \p_x^{k-1}w\Vert_{L^2}^2+1\big)
\end{align*}
3. The nonlinearity term:
\begin{align*}
\la \int_\T \p_x^{k-1}\big(\rho^{-\alpha} w^2\big)\p_x^{k-1}w\ra &=\la \int_\T \p_x^{k-2}\big(\rho^{-\alpha} w^2\big)\p_x^kw\ra \\
&\le  C\Vert \rho^{-\alpha}\Vert_{H^{k-2}}\Vert w\Vert_{H^{k-2}}^2\Vert \p_x^kw\Vert_{L^2}\\
&\le  \frac{1}{20c}\Vert \p_x^kw\Vert_{L^2}+C'\Vert \rho^{-\alpha}\Vert_{H^{k-2}}^2\Vert w\Vert_{H^{k-2}}^4\\
&\le \frac{1}{20c}\Vert \p_x^kw\Vert_{L^2}+M\big(E_{k}, \Vert f\Vert_{L^2(0, T; H^{k-2})}, \frac{1}{ \cl }, T\big).\end{align*}
4. The zero order term:
\begin{align*}
\la \int_\T \p_x^{k-1}(\rho^{2\gamma-\alpha})\p_x^{k-1}w\ra &\le C\Vert \rho^{2\gamma-\alpha}\Vert_{H^{k-1}}\Vert \p_x^{k-1}w\Vert_{L^2}\\
&\le M\big(E_{k}, \Vert f\Vert_{L^2(0, T; H^{k-2})}, \frac{1}{ \cl }, T\big)\Vert \p_x^{k-1}w\Vert_{L^2}.
\end{align*}
5. The forcing term:
\begin{align*}
\la \int_{\T} \p_x^{k-1}\big(\rho^\alpha \p_xf\big)\p_x^{k-1}w\ra&= \la \int_{\T} \p_x^{k-2}\big(\rho^\alpha \p_xf\big)\p_x^kw\ra\\
&\le C\Vert \rho^\alpha\Vert_{H^{k-2}}\Vert\p_xf\Vert_{H^{k-2}}\Vert \p_x^kw\Vert_{L^2}\\
&\le \frac{1}{20c}\Vert \p_x^kw\Vert_{L^2}^2+M\big(E_{k}, \Vert f\Vert_{L^2(0, T; H^{k-2})}, \frac{1}{ \cl }, T\big)\Vert f\Vert_{H^{k-1}}^2.
\end{align*}
Putting the estimates 1. through 5. together, we obtain
\begin{align*}
\mez\frac{\de}{\de t}\Vert \p_x^{k-1}w\Vert_{L^2}^2&\le \frac{-2}{5c}\Vert \p_x^kw\Vert_{L^2}^2+M\big(E_{k}, \Vert f\Vert_{L^2(0, T; H^{k-2})}, \frac{1}{ \cl }, T\big)\Vert \p_x^{k-1}w\Vert_{L^2}^2\\
&\quad+M\big(E_{k}, \Vert f\Vert_{L^2(0, T; H^{k-2})}, \frac{1}{ \cl }, T\big)\big(\Vert f\Vert_{H^{k-1}}^2+1).
\end{align*}
Combining this with \eqref{dtrhoHk} and Gr\"onwall's lemma leads to
\begin{align*}
&\Vert \p_x^k\rho\Vert_{L^\infty(0, T; L^2)}^2+\Vert \p_x^{k-1}w\Vert_{L^\infty(0, T; L^2)}^2+\Vert \p_x^kw\Vert_{L^2(0, T; L^2)}^2\\
&\quad \le M\Big(\Vert \p_x^k\rho_0\Vert_{L^2}^2+\Vert \p_x^{k-1}w_0\Vert_{L^2}^2+\Vert f\Vert_{L^2(0, T; H^{k-1})}^2+T\Big)\exp\Big(M\big(\Vert \p_x^{k} u\Vert_{L^1(0, T; L^2)}+T\big)\Big)
\end{align*}
where we denoted
\[
M\equiv M\big(E_{k}, \Vert f\Vert_{L^2(0, T; H^{k-2})}, \frac{1}{ \cl }, T\big)
\]
and used the fact that the $L^2(0, T; H^k)$ norm of $u$ is controlled by $M$.

It follows easily from this that $\Vert \partial_x^{k}u\Vert_{L^\infty(0, T; L^2)}$ and $\Vert \partial_x^{k+1}u\Vert_{L^2(0, T; L^2)}$ can be controlled by the same bound. This finishes the proof of \eqref{w:Hk-1}.
\end{proof}
 In view of Lemmas \ref{lemm:wL2}, \ref{lemm:wH1} and \ref{lemm:wH2} we have proved that
 \bq 
\begin{aligned}
&\sup_{T\in [0,T^*)} \Vert \rho\Vert_{L^\infty(0, T;  H^k)}+\sup_{T\in [0,T^*)}\Vert u\Vert_{L^\infty(0, T; H^k)}+\sup_{T\in [0,T^*)}\Vert u\Vert_{L^2(0, T; H^{k+1})}\\
&\qquad\qquad\qquad\qquad\le M_k\Big(\Vert (\rho_0, u_0)\Vert_{H^k\times H^k}, \Vert f\Vert_{L^2(0, T^*; H^{\max\{k-1, 1\}})}, \frac{1}{ \cl }, T^*\Big)<\infty
\end{aligned}
\eq
for $k\ge 1$. 
 Appealing to local existence, established by Prop. \ref{prop:local}, the solution can be extended past $T^*$.

\section{Proof of Theorem \ref{theo:global}}
We assume here that $c_p >0$ and that $\alpha\in (\frac{1}{2},1]$, $\gamma\geq 2\alpha$. By Prop. \ref{prop:local}, there exists a positive time $T_0$ such that problem \eqref{eq:mass}-\eqref{eq:initv} has a unique solution $(\rho, u)$ on $[0, T_0]$ such that
 \bq\label{localreg:v}
\rho\in C(0, T_0; H^k),\quad u\in C(0, T_0; H^k)\cap L^2(0, T_0; H^{k+1}), \quad k\ge 3,
\eq
and $\rho>0$ on $[0, T_0]$. 
 Let $T^*$ be the maximal lifetime of the classical solution $(\rho,u)$, so that, by Thm. \ref{theo:cont}, 
\bq\label{assu:contr}
\inf_{t\in (0, T^*)}\min_{x\in \T}\rho(x, t)=0.
\eq
We claim that $T^*=\infty$. We will argue by contradiction.    Let us note that the $H^k$ regularity, $k\ge 3$, of $(\rho, u)$ suffices to justify all the calculations below. Recall from the proof of Lemma \ref{entropy:0} in Appendix~\ref{app:entropy}, that
\be
X= u + c_\mu \rho^{\alpha-2}\partial_x \rho,
\ee
defined also in Eq.~\eqref{xDef}, satisfies
\begin{align}\label{DtX}
\p_t  X + u \p_x X & 
= - \gamma \frac{c_p }{c_\mu} \rho^{\gamma-\alpha}  (X-u)+f=- \gamma \frac{c_p }{c_\mu} \rho^{\gamma-\alpha}X+\gamma \frac{c_p }{c_\mu}\rho^{\gamma-\alpha}u+f.
\end{align}
By Lemma \ref{PiCont:1} 1., we have 
\be\label{gs1:rhoLinfty}
\|\rho\|_{L^\infty(0,T;L^\infty(\T))} \leq M(E_1, \Vert f\Vert_{L^2(0, T; L^\infty)}, T).
\ee
Since $\gamma\geq 2\alpha\geq \alpha+\frac{1}{2}$ for $\alpha\in (\frac{1}{2},1]$, combining the above estimate with \eqref{enCont1}, we have
\be\label{gs1:f1}
\| \rho^{\gamma-\alpha} u\|_{L^\infty(0,T;L^2(\T))}  \leq M(E_1, \Vert f\Vert_{L^2(0, T; L^\infty)}, T).
\ee
Note also
\begin{align*}
\partial_x(\rho^{\gamma-\alpha} u)
&= (\sqrt{\rho}\partial_x u) \rho^{\gamma-\alpha-\frac{1}{2}} + (\gamma-\alpha)   \rho^{\gamma-2\alpha} (\rho^{\alpha-\frac{3}{2}}\partial_x \rho) (\sqrt{\rho} u) 
\end{align*}
Now, estimate \eqref{entinqe1} implies 
\[
\| (\rho^{\alpha-\frac{3}{2}}\partial_x \rho)\|_{L^2(0,T;L^2(\T))}  \leq M(E_1, \Vert f\Vert_{L^2(0, T; L^\infty)}, T).
\]
Putting together this,~\eqref{enCont1},~\eqref{enCont2},~\eqref{gs1:rhoLinfty}, and the assumption that $\gamma\ge 2\alpha$ we deduce that
\[
\|\partial_x(\rho^{\gamma-\alpha} u)\|_{L^2(0,T;L^1(\T))}  \leq M(E_1, \Vert f\Vert_{L^2(0, T; L^\infty)}, T).
\]
which combined with \eqref{gs1:f1} yields
\be\label{est-rhs-Xeq}
\Vert \rho^{\gamma-\alpha} u\Vert_{L^2(0,T;W^{1, 1})}\le M(E_1, \Vert f\Vert_{L^2(0, T; L^\infty)}, T).
\ee
Since \eqref{DtX} is a transport equation we then have
\be\label{X:Linfty}
\begin{aligned}
	\Vert X \Vert_{L^\infty(0,T; L^\infty)} &\le \big(\Vert X_0 \Vert_{L^\infty}+\gamma \frac{c_p }{c_\mu}\Vert \rho^{\gamma-\alpha}u\Vert_{L^1(0, T; L^\infty)}+\Vert f\Vert_{L^1(0, T; L^\infty)}\big)\exp\big(\gamma \frac{c_p }{c_\mu} \Vert\rho^{\gamma-\alpha}\Vert_{L^1(0, T; L^\infty)}\big)\\
	&\le M(E_1, \Vert X_0\Vert_{L^\infty},\Vert f\Vert_{L^2(0, T; L^\infty)}, T).
\end{aligned}
\ee
Recall that $X=u + \frac{\partial_x \rho}{\rho^2} \mu(\rho)=u + c_\mu\rho^{\alpha-2}\p_x\rho$, hence $X\rho^{\gamma-\alpha}=u\rho^{\gamma-\alpha}+ c_\mu\rho^{\gamma-2}\p_x\rho$. It then follows from \eqref{gs1:rhoLinfty}, \eqref{est-rhs-Xeq} and \eqref{X:Linfty} that
\bq\label{est:rhs:veq}
\Vert \rho^{\gamma-2}\p_x\rho\Vert_{L^2(0, T; L^{\infty})}\le  M(E_1, \Vert X_0\Vert_{L^\infty},\Vert f\Vert_{L^2(0, T; L^\infty)}, T).
\eq
Using \eqref{eq:mass} and \eqref{eq:mom} we obtain
\bq\label{eq:v}
\p_t u+(u-\frac{\mu'(\rho)\p_x\rho}{\rho})\p_xu=\frac{\mu(\rho)}{\rho}\p_x^2u-\frac{p'(\rho)\p_x\rho}{\rho}+f=c_\mu \rho^{\alpha-1}\p_x^2u-c_p \gamma\rho^{\gamma-2}\p_x\rho+f.
\eq
Using the maximum principle (see the argument leading to \eqref{eq:wm} below and a similar argument for the minimum) and the bound \eqref{est:rhs:veq} gives
\begin{equation}\label{eq:ulinf}
\begin{aligned}
\Vert u\Vert_{L^\infty(0, T; L^{\infty})}&\le \Vert u_0\Vert_{L^\infty}+c_p  \gamma\Vert\rho^{\gamma-2}\p_x\rho\Vert_{L^1(0, T; L^\infty)}+\Vert f\Vert_{L^1(0, T; L^\infty)}\\
&\le M(E_1, \Vert (X_0, u_0)\Vert_{L^\infty},\Vert f\Vert_{L^2(0, T; L^\infty)}, T).
\end{aligned}
\end{equation}
From the definition of $X$ and \eqref{X:Linfty}, this yields
\bq\label{rho:key1}
\Vert \p_x\rho^{\alpha-1}\Vert_{L^\infty(0, T ; L^{\infty})}\le M(E_1, \Vert (X_0, u_0)\Vert_{L^\infty},\Vert f\Vert_{L^2(0, T; L^\infty)}, T)
\eq
when $\alpha< 1$, and
\bq\label{rho:key2}
\Vert \p_x\ln \rho\Vert_{L^\infty(0, T ; L^{\infty})}\le M(E_1, \Vert (X_0, u_0)\Vert_{L^\infty},\Vert f\Vert_{L^2(0, T; L^\infty)}, T)
\eq
when $\alpha=1$.

When $\alpha< 1$, the continuity equation implies
\be
\partial_t (\rho^{\alpha-1}) =  -(\alpha-1)\partial_x (u\rho)  \rho^{\alpha-2}.
\ee
Integrating this in space and time and using the definition of $X$ leads to
\bq\begin{aligned}\label{rho:key11}
\int_\T \rho^{\alpha-1}(x,T) \de x&=\int_\T\rho_0^{\alpha-1}\de x+(\alpha-1)(\alpha-2)\int_0^t \int_\T (u\rho\rho^{\alpha-3}\p_x \rho) (x,z)\de x \de z\\
&=\int_\T\rho_0^{\alpha-1}\de x+\frac{1}{c_\mu}(\alpha-2)(\alpha-1)\int_0^t \int_\T(uc_\mu\rho^{\alpha-2}\partial_x\rho)(x,z)\de x \de z\\
&\le \int_\T\rho_0^{\alpha-1}\de x+C\int_0^t \int_\T X^2(x,z)\de x \de z,
\end{aligned}
\eq
valid for $0 \leq t \leq T$.

Similarly, when $\alpha=1$ we have
\be \label{eq:intlog}
\left| \int_\T \ln\rho(x,t)\de x \right| \le  \left|\int_\T\ln\rho_0\de x \right|+C\int_0^t \int_\T X^2(x,z) \de x \de z,  \quad 0 \leq t \leq T.
\ee
Then by virtue of \eqref{X:Linfty}, {\eqref{eq:ulinf}}, \eqref{rho:key1}, \eqref{rho:key11}, Poincar\'e-Wirtinger's inequality and Sobolev embedding we deduce that
\[
\Vert \rho^{\alpha-1}\Vert_{L^\infty(0, T; L^\infty)}\le M(E_1, \Vert (X_0, u_0)\Vert_{L^\infty}, \Vert \rho_0^{\alpha -1}\Vert_{L^1}, \Vert f\Vert_{L^2(0, T; L^\infty)}, T)
\]
if $\alpha< 1$.

On the other hand, if $\alpha =1$,~\eqref{gs1:rhoLinfty} combined with with ~\eqref{eq:intlog}, Poincar\'e-Wirtinger's inequality and Sobolev embedding, yields
\[
\Vert \ln \rho\Vert_{L^\infty(0, T; L^\infty)} \leq M(E_1, \Vert (X_0, u_0)\Vert_{L^\infty}, \Vert \ln \rho_0\Vert_{L^1}, \Vert f\Vert_{L^2(0, T; L^\infty)}, T).
\]
Consequently
\[
\inf_{(x, t)\in \T\times [0, T]}\rho(x, t)\ge \mathcal{F}\left(M(E_0, \Vert (X_0, u_0)\Vert_{L^\infty}, \Vert  \rho_0^{\alpha-1}\Vert_{L^1}+\Vert \ln \rho_0\Vert_{L^1}, \Vert f\Vert_{L^2(0, T; L^\infty)}, T)\right)
\]
where 
\be
\mathcal{F}(z)=
\begin{cases}
	z^{\frac{1}{\alpha-1}}&\quad\text{if}~\alpha< 1,\\
	e^{-z}&\quad\text{if}~\alpha= 1.
\end{cases}
\ee
Therefore,
\[
\inf_{(x, t)\in \T\times [0, T^*)}\rho(x, t)\ge \mathcal{F}\left(M(E_0, \Vert (X_0, u_0)\Vert_{L^\infty}, \Vert  \rho_0^{\alpha-1}\Vert_{L^1}, \Vert \ln \rho_0\Vert_{L^1}, \Vert f\Vert_{L^2(0, T^*; L^\infty)}, T^*)\right)>0
\]
which contradicts \eqref{assu:contr}.

\section{Proof of Theorem \ref{theo:global2}}
Recall the assumptions \eqref{assumThm1.2} and \eqref{assumThm1.2:2}
Assume that $c_p>0$ and either 
\begin{align}
&\alpha>\mez, \quad \gamma\in [\alpha,\alpha+1], \quad \gamma \neq 1\quad\text{or}\\ 
&\alpha\ge 0,\quad \gamma\in [\alpha,\alpha+1],\quad \gamma>1.
\end{align}
 By Prop. \ref{prop:local}, there exists a positive time $T_0$ such that problem \eqref{eq:mass}-\eqref{eq:initv} has a unique solution $(\rho, u)$ on $[0, T_0]$ such that
 \bq\label{localreg:v}
\rho\in C(0, T_0; H^k),\quad u\in C(0, T_0; H^k)\cap L^2(0, T_0; H^{k+1}), \quad k\ge 4,
\eq
and $\rho>0$ on $[0, T_0]$.  Let $T^*$ be the maximal existence time. We claim that $T^*=\infty$. Assume by contradiction that $T^*$ is finite. By Theorem \ref{theo:cont} we have
 \bq\label{contradiction:assu}
 \inf_{t\in [0, T^*)} \min_{x\in \T}\rho(x, t)=0.
 \eq
From Lemma \ref{LemwEqn}, the $w$ equation \eqref{Weqn} is
\begin{align} \nonumber
\partial_t w&=  c_\mu\rho^{\alpha-1} \partial_x^2 w- (u+c_\mu\rho^{\alpha-2}\partial_x \rho )\partial_xw+    \frac{c_p }{c_\mu } \left(\gamma -2 (\alpha+1)  \right)\rho^{\gamma-\alpha}w\\
&\quad -\frac{1}{c_\mu} (\alpha+1)\rho^{-\alpha}   w^2+ \frac{c^2_p }{c_\mu } \left(\gamma - (\alpha+1) \right)\rho^{2\gamma-\alpha}.
\label{wEqn2}
\end{align}
Note that the assumption $f(x, t)=f(t)$ was used to have $\p_xf=0$.
 It follows from \eqref{localreg:v} and the equation \eqref{wEqn2} that 
\[
w\in C(0, T; H^3)\cap L^2(0, T; H^4), \qquad \partial_t w \in C(0,T; H^1)\subset C(\T \times [0,T])
\]
 Thus, $w\in C^1(\T\times [0, T])$ and thus the function
\be\label{wmax}
w_M(t):=\max_{x\in \T} w(x, t)
\ee
is Lipschitz continuous on $[0, T]$.  According to the Rademacher theorem, $w_M$ is differentiable almost everywhere on $[0 ,T]$.  
There exists for each $t\in [0, T^*)$ a point $x_t$ such that 
\[
w_M(t)=w(x_t, t).
\]
Let $t\in (0, T)$ be a point at which $w_M$ is differentiable. We have
\begin{align*}
w'_M(t)&=\lim_{h\to 0^+}\frac{w_M(t+h)-w_M(t)}{h}\\
&=\lim_{h\to 0^+}\frac{w(x_{t+h}, t+h)-w(x_t, t)}{h}\\
&\ge \lim_{h\to 0^+}\frac{w(x_t, t+h)-w(x_t, t)}{h}=\p_tw(x_t, t).
\end{align*}
On the other hand,
\begin{align*}
w'_M(t)&=\lim_{h\to 0^+}\frac{w_M(t)-w_M(t-h)}{h}\\
&=\lim_{h\to 0^+}\frac{w(x_t, t)-w(x_{t-h}, t-h)}{h}\\
&\le \lim_{h\to 0^+}\frac{w(x_t, t)-w(x_t, t-h)}{h}=\p_tw(x_t, t).
\end{align*}
Thus, $w_M'(t)=\p_tw(x_t, t)$ if $w_M$ is differentiable at $t$. We deduce from this and equation \eqref{wEqn2} that for almost every $t\in (0, T)$,
\begin{align}\label{eq:wm}
\partial_t w_M&\leq     A(t) w_M  + B(t) w_M^2+ C(t) 
\end{align}
with
\begin{align*}
A(t)&:= \frac{c_p }{c_\mu } \left(\gamma -2 (\alpha+1)  \right)\rho(x_t)^{\gamma-\alpha}\\
B(t)&:= -\frac{1}{c_\mu}(\alpha+1)\rho(x_t)^{-\alpha}\\
C(t)&:=\frac{c^2_p }{c_\mu }  \left(\gamma - (\alpha+1) \right)\rho(x_t)^{2\gamma-\alpha}.
\end{align*}
where we used the facts that $\p_x^2 w(x_t, t)\le 0$ and $\p_x w(x_t, t)=0$.  Note that $B(t)\le 0$. In addition, the function $C$ is nonpositive under the conditions \eqref{assumThm1.2}. The condition on the initial data \eqref{initialDataCond} is equivalent to $w_M(0)\leq 0$.  We deduce that
\be\label{maximumPrinc}
w(t)\leq 0, \qquad \forall t<T^*.
\ee
At the point $y_t$ where the density attains its minimum value $\rho_m:=\rho(y_t,t)$, $\rho_m$ satisfies
\begin{align}
\partial_t \rho_m &= - \partial_x u(y_t) \rho_m = - \frac{w(y_t)}{c_\mu} \rho_m^{1-\alpha} - \frac{c_p }{c_\mu} \rho_m^{\gamma -\alpha+ 1}\geq - \frac{c_p }{c_\mu} \rho_m^{\gamma -\alpha+ 1}
\end{align}
where we used \eqref{maximumPrinc}.  Provided that $\gamma\neq \alpha$, this  implies the differential inequality
\begin{align}
\frac{1}{(\alpha -\gamma)}\partial_t ( \rho_m^{\alpha -\gamma}
)\geq - \frac{c_p }{c_\mu}.
\end{align}
  Since  $\alpha<\gamma$, we find
\begin{align}
\partial_t ( \rho_m^{\alpha -\gamma}
)\leq  \frac{c_p }{c_\mu}(\gamma-\alpha)
\end{align}
which implies 
\be \label{firstmin}
\rho_m(t)\geq  \left(\rho_m(0)^{\alpha -\gamma} + t\frac{c_p }{c_\mu}(\gamma-\alpha)\right)^{\frac{1}{\alpha-\gamma}}, \quad\forall t<T^*
\ee
Since $c_p /c_\mu>0$, this implies that 
\be
 \inf_{t\in [0, T^*)} \min_{x\in \T}\rho(x, t)\geq  \left(\rho_m(0)^{\alpha -\gamma} + T^*\frac{c_p }{c_\mu}(\gamma-\alpha)\right)^{{\frac{1}{\alpha-\gamma}}}>0
\ee 
which contradicts the assumption \eqref{contradiction:assu}.    We conclude that the solution $(\rho, u)$ is global in time. 

On the other hand, when $\alpha = \gamma$ we have
\begin{align}
\partial_t \ln \rho_m
\geq - \frac{c_p }{c_\mu}
\end{align}
and thus
\be\label{secondmin}
\rho_m(t)\geq \rho_m(0)\exp\left(-t \frac{c_p}{c_\mu}\right) >0
\ee
which again leads to a contradiction with \eqref{contradiction:assu}.

\begin{rema}
With a more refined maximum principle argument, one can relax the regularity requirement of $k\geq 4$ which we used to conclude that \eqref{wmax} is Lipschitz continuous on $[0,T]$.
\end{rema}

\section{Proof of Theorem \ref{densBndProp}}
In this section, we give an upper bound for the long-time average maximum density, assuming that the forcing has zero mean in space. This follows by an application of the Bresch-Desjardins's entropy and the following elementary lemma.
\begin{lemm}\label{lem:easy}
	Let {$m\geq \frac 12$}. If $h^m\in W^{1,1}(\T)$ then we have
	\begin{equation}\label{interlemm}
	\Vert h \Vert_{L^{\infty}(\T)} \leq 2 \Vert \p_x (h^{m}) \Vert^{\frac 1 m}_{L^1(\T)} + 4 \Vert h \Vert_{L^1(\T)}.
	\end{equation}
\end{lemm}
\begin{proof}[Proof of Lemma~\ref{lem:easy}]
	Since ${h}\in W^{1,1}(\T)\subset C^{0}(\T)$,  we have $h\in C^0(\T)$. In particular, there exists a point $x_0 \in \T$ such that $|h(x_0)| \leq {\sqrt{2}}\Vert h \Vert_{L^1(\T)}$. For all $x\in \T$ we have
	\[
	h^m(x)=\int_{x_0}^x\p_y(h^m(y))dy+h^m(x_0),
	\]
	hence
	\[
	|h(x)|^m\le \Vert \p_xh^m\Vert_{{L^1}(\T)}+|h(x_0)|^m\le \Vert \p_x(h^m)\Vert_{L^1(\T)}+{\sqrt{2}}\Vert h\Vert_{L^1(\T)}^m.
	\]
	In view of the elementary inequality 
	\[
	(a+b)^{\frac 1 m} \leq 2a^{\frac 1 m} + 2b^{\frac 1 m},\quad a,~b,~m>0,
	\] 
	we thus obtain \eqref{interlemm}.
\end{proof}
\begin{proof}[Proof of Theorem \ref{densBndProp}] Recall our assumptions
	\begin{equation}\label{LongTimeAssum}
	\gamma \in [\max\{2-\alpha,\alpha\}, \alpha + 1], \quad \alpha \geq 1/2, \quad \text{and} \quad c_p, c_\mu > 0.
	\end{equation}
	
	Next, by Lemma~\ref{entropy:0}, the entropy
		\be
	s= \frac{\rho}{2} \left|u + \frac{\partial_x \rho}{\rho^2} \mu(\rho)\right|^2 + \pi (\rho).
	\ee
	satisfies
	\begin{equation}
	\frac{\de}{\de t}\int_{\T}  s(x,t) \de x  = -\int_{\T} |\partial_x \rho|^2 \mu(\rho) \frac{p'(\rho)}{\rho^2} dx+\int_\T  f\rho \big(u + \frac{\partial_x \rho}{\rho^2}\mu(\rho)\big)\de x.
	\end{equation}
	Integrating this in time yields
	\begin{equation*}
	\begin{aligned}
	\int_{\T} s(x,T)\de x - \int_{\T} s(x,0)\de x + &c_p c_\mu \gamma \int_0^T \int_\T \rho^{\alpha + \gamma -3} |\p_x \rho|^2 \de x \de t \\
	&= \int_0^T \int_{\T} f \rho u \de x \de t + c_\mu \int_0^T \int_\T f \rho^{\alpha -1 } \p_x \rho \, \de x \de t.
	\end{aligned}
	\end{equation*}
	Using the assumption \eqref{zeromean} we calculate
	\begin{equation*}
	\begin{aligned}
	\int_0^T \int_{\T} f \rho u \, \de x \de t &= - \int_0^T \int_{\T} g \p_x(\rho u) \de x \de t = \int_0^T \int_{\T} g \p_t \rho \, \de x \de t \\
	&= \int_\T( g \rho)(x,T) \, \de x - \int_\T (g \rho)(x,0) \, \de x - \int_0^T \int_{\T} \rho \p_t g \, \de x \de t.
	\end{aligned}
	\end{equation*}
	This implies
	\begin{equation*}
	\begin{aligned}
	&\left|\int_0^T \int_{\T} f \rho u \, \de x \de t \right| \leq 2\Vert g \Vert_{L^\infty(0,T; L^\infty)}\Vert \rho_0 \Vert_1 +  \Vert\p_t g \Vert_{L^1(0,T; L^\infty)}\Vert \rho_0 \Vert_1\\
	&\leq 2\Vert g \Vert_{L^\infty(0,T; L^\infty)}\Vert \rho_0 \Vert_1 +  T \Vert\p_t g \Vert_{L^\infty(0,T; L^\infty)}\Vert \rho_0 \Vert_1.
	\end{aligned}
	\end{equation*}
	On the other hand, using Cauchy--Schwarz, we have
	\begin{equation*}
	\begin{aligned}
	\left|c_\mu \int_0^T \int_\T f \rho^{\alpha -1 } \p_x \rho \, \de x \de t \right| &\leq \frac 12 c_p c_\mu \gamma \int_0^T \int_\T \rho^{\alpha + \gamma - 3}|\p_x \rho|^2 \de x \de t + C \int_0^T \int_\T \rho^{\alpha -\gamma +1} f^2 \de x \de t\\
		&\leq  \frac 12 c_p c_\mu \gamma \int_0^T \int_\T \rho^{\alpha + \gamma - 3}|\p_x \rho|^2 \de x \de t + CT(1+\Vert \rho_0 \Vert_1)\Vert f\Vert_{L^\infty(0, T; L^\infty)}^2.
	\end{aligned}
	\end{equation*}
	Here, $C$ is a constant which depends only on $c_\gamma, c_p$ and $\gamma$. We have used the assumption \eqref{LongTimeAssum} that $\gamma$ belongs to the range $ \gamma\in [\max\{2-\alpha,\alpha\}, \alpha+1]$ with $\alpha\geq 1/2$ to have $0 \leq \alpha - \gamma +1 \leq 1$.
	
	Note that the allowed range of $\gamma$ and $\alpha$ requires that $\gamma\geq 3/2$ always.  Since, in particular $\gamma>1$ we have $\pi(\rho)\ge 0$ and $s\ge 0$. Thus, putting all together, we obtain the bound
	\begin{equation*}
	\begin{aligned}
\frac 12& c_p c_\mu \gamma \int_0^T \int_\T \rho^{\alpha + \gamma - 3}|\p_x \rho|^2 \de x \de t \\
	&\leq  2\Vert g \Vert_{L^\infty(0,T; L^\infty)}\Vert \rho_0 \Vert_1 +  T \Vert\p_t g \Vert_{L^\infty(0,T; L^\infty)}\Vert \rho_0 \Vert_1+ CT(1+\Vert \rho_0 \Vert_1 )\Vert \p_xg\Vert_{L^\infty(0, T; L^\infty)}^2+ \int_\T s(x,0) \de x.
	\end{aligned}
	\end{equation*}
	We thus obtain 
	\[
	\frac 12c_p c_\mu \gamma \int_0^T \int_\T \rho^{\alpha + \gamma - 3}|\p_x \rho|^2 \de x \de t\leq M_1 T + M_0,
	\]
	where $M_0$ is a constant which depends only on $c_\mu$, $c_p$, $\gamma$, $\alpha$, $\Vert\rho_0 \Vert_{L^\infty}$, $\Vert\rho^{-1}_0 \Vert_{L^\infty}$, $\Vert u_0\Vert_{L^2}$, $\Vert \p_x \rho_0 \Vert_{L^2}$, $\Vert g \Vert_{L^\infty(0,T; L^\infty)}$, and $M_1$ a constant which depends only on $c_\mu$, $c_p$, $\gamma$, $\Vert\rho_0 \Vert_{L^1}$, $\Vert \p_t g \Vert_{L^\infty(0,T; L^\infty)}$, $\Vert \p_xg\Vert_{L^\infty(0, T; L^\infty)}$.
	
	In particular,
	\begin{equation*}
	\int_0^T \int_\T |\p_x (\rho^{\frac 1 2(\alpha + \gamma - 1)})|^2 \de x \de t  \leq M_3 T + M_2,
	\end{equation*}
	where $M_{i +2} = \frac{(\alpha + \gamma -1)^2}{2 c_p c_\mu \gamma}M_i$, for $i = 0,1$. Here, we used the fact that $\alpha + \gamma -1 > 0$.
	
	{By assumption \eqref{LongTimeAssum} we have that $\alpha + \gamma \geq 2\max\{1,\alpha\} \geq 2$ which implies $\frac 1 m \leq 2$.} We now apply Lemma~\ref{lem:easy} with $m := \frac 1 2 (\alpha + \gamma -1)$. Using the embedding $L^2(\T)\subset L^1(\T)$, we obtain
	\begin{equation*}
	\int_0^T \Vert \rho(\cdot, t) \Vert_{L^\infty}\de t \leq 2\int_0^T \Vert \p_x (\rho^{m}) \Vert^{\frac 1 m}_{L^2}\de t + 4 T \Vert \rho_0 \Vert_{L^1}.
	\end{equation*}
	{Consequently,}
	\begin{equation*}
	\int_0^T \Vert \rho(\cdot, t) \Vert_{L^\infty}\de t \leq 2\int_0^T (\Vert \p_x (\rho^{m}) \Vert^{2}_{L^2} + 1)\de t + 4 T \Vert \rho_0 \Vert_{L^1} \leq 2 (M_3 T + M_2) +  2T + 4 T \Vert \rho_0 \Vert_{L^1}.
	\end{equation*}
	Hence,
	\begin{equation}
	\frac 1 T \int_0^T \Vert \rho(\cdot, t) \Vert_{L^\infty}\de t \leq (2M_3+ 2 + 4\Vert \rho_0 \Vert_{L^1}) + \frac{2}{T} M_2,
	\end{equation}
	and the claim follows, with the definition
	\begin{equation}\label{eq:m4m5}
	C_1 = 2 M_2, \qquad  C_2 := 2M_3+ 2 + 4\Vert \rho_0 \Vert_{L^1}.
	\end{equation}
	
\end{proof}

\appendix

\section{Bresch-Desjardins's entropy}\label{app:entropy}

For the sake of completeness we present the proof of Lemma \ref{entropy:0} which essentially follows from \cite{Bresch2003, Bresch2003b, BreschReview}. From the continuity equation \eqref{eq:mass}, any smooth $\xi(\rho)$ satisfies
\be\label{MassarbFunction}
\partial_t \xi(\rho) = \partial_t \rho \xi'(\rho) = -\partial_x(u\rho)  \xi'(\rho)= -u \partial_x \xi(\rho) - \rho (\partial_x u) \xi'(\rho)
\ee
Using equation \eqref{MassarbFunction} applied to the function $\partial_x \xi(\rho)$, we find the evolution of $\rho \partial_x \xi(\rho))$:
\bq
\begin{aligned}
\partial_t(\rho \partial_x \xi(\rho)) &= -\partial_x(\rho u) \partial_x \xi(\rho) + \rho \partial_t \partial_x \xi(\rho)\\
&= -\partial_x(\rho u) \partial_x \xi(\rho) - \rho\partial_x( u \partial_x \xi(\rho) + \rho (\partial_x u) \xi'(\rho))\\
&= -\partial_x(\rho u) \partial_x \xi(\rho) - \rho \partial_xu \partial_x \xi(\rho)- \rho u \partial_x^2 \xi(\rho)- \rho\partial_x(\rho (\partial_x u) \xi'(\rho))\\
&= -\partial_x(\rho u \partial_x \xi(\rho)) - \rho \partial_xu \partial_x \xi(\rho)- \rho\partial_x(\rho (\partial_x u) \xi'(\rho))\\
&= -\partial_x(\rho u \partial_x \xi(\rho)) - \partial_x(\rho^2 (\partial_x u) \xi'(\rho)). \label{arbDerEq}
\end{aligned}
\eq
Then, letting $X:= u+ \partial_x \xi(\rho)$, combining Eq. \eqref{arbDerEq} with the momentum equation \eqref{eq:mom} yields
\begin{align}\label{xeq1}
\partial_t (\rho X) = -\partial_x(\rho u X) - \partial_x p(\rho) + \partial_x (\mu(\rho) \partial_x u) - \partial_x (\rho^2 (\partial_x u)\xi'(\rho))+\rho f.
\end{align}
We now choose $\rho^2\xi'(\rho)= \mu(\rho)$, so that the final two terms in \eqref{xeq1} cancel. Thus with this choice,
\be\label{xDef}
X= u + \frac{\partial_x \rho}{\rho^2} \mu(\rho)
\ee
and, by \eqref{xeq1}, $\rho X$ satisfies
\begin{align}
\partial_t (\rho X) = -\partial_x(\rho u X) - \partial_x p(\rho)+\rho f.
\end{align}
Whence, we obtain
\begin{align}
\partial_t (\rho X^2) = -\partial_x(\rho u X^2) - 2X\partial_x p(\rho)+2\rho f X.
\end{align}
Integrating in space
\begin{align}  \nonumber
	\frac{1}{2} \frac{\de}{\de t}\int_{\T}  (\rho X^2) (x,t) \de x  
	&= -\int_{\T} \rho u\frac{\partial_x p(\rho)}{\rho}\de x -\int_{\T} |\partial_x \rho|^2 \mu(\rho) \frac{p'(\rho)}{\rho^2} \de x+\int_\T  f\rho \big(u + \frac{\partial_x \rho}{\rho^2}\mu(\rho)\big)\de x\\ \nonumber
		&= -\int_{\T} \rho u\  \partial_x \pi'(\rho)\de x -\int_{\T} |\partial_x \rho|^2 \mu(\rho) \frac{p'(\rho)}{\rho^2} \de x+\int_\T  f\rho \big(u + \frac{\partial_x \rho}{\rho^2}\mu(\rho)\big)\de x\\\nonumber
		&=- \frac{\de}{\de t}\int_{\T}  \pi(\rho) \de x -\int_{\T} |\partial_x \rho|^2 \mu(\rho) \frac{p'(\rho)}{\rho^2} \de x+\int_\T  f\rho \big(u + \frac{\partial_x \rho}{\rho^2}\mu(\rho)\big)\de x.
		\end{align}
The global balance \eqref{entropyBalance} for entropy $s:=\frac{1}{2}\rho X^2+\pi(\rho)$  follows.

\section{Local well-posedness}\label{appendix:Loc}

\begin{prop}\label{prop:local}
	Assume that $p:\mathbb{R}^+\to \mathbb{R}$ and $\mu:\mathbb{R}^+\to \mathbb{R}^+$ are $C^\infty$ functions away from zero.
	Let $\rho_0$ and $u_0$ belong to $H^k(\T)$ for an integer $k \geq 1$, such that $r_0 := \min_{x\in \T}\rho_0 > 0$. Suppose that for all $T>0$
	\begin{equation*}
	f\in L^2(0,T; H^{k-1}(\T)).
	\end{equation*}
	Then, there exists a $T_0>0$ depending only on $\Vert (\rho_0, u_0)\Vert_{H^k(\T)\times H^k(\T)}$, { $r_0$ and $f$}, and a unique strong solution $(\rho, u)$ to \eqref{eq:mass}-\eqref{eq:initv} on $[0, T_0]$ with data $(\rho_0,u_0)$ such that 
	\[
	\rho\in C(0, T_0; H^k(\T)),\quad u\in C(0, T_0; H^k(\T))\cap L^2(0, T_0; H^{k+1}(\T))
	\]
	and $\rho(x, t)>\frac{r_0}{2}$ for all $(x, t)\in \T\times [0,T_0]$.
\end{prop}

\begin{proof}
	{\bf Step 0.} (Iteration Scheme) {We are going to set up an iteration argument and prove that the iterates converge to the desired solution.} 
	Let us first suppose that the initial data $\rho_0, u_0$ are smooth, and let us define $r_0 := \min_{x \in \T} \rho_0$.
	
	Let us initialize our scheme as follows:
	\begin{equation*}
	\begin{aligned}
	&(\rho_{0}(x,t), u_{0}(x,t)) := (\rho_0 (x), u_0(x)),\\
	&\rho_1(x,t) = \rho_0(x),
	\end{aligned}
	\end{equation*}
	and we define $u_1(x,t)$ so that
	\begin{equation}\label{eq:initu1}
	\begin{aligned}
	&\p_t u_{1} - \frac{\mu(\rho_{1})}{\rho_{1}} \p_x^2 u_{1} = - u_{0} \p_x u_{0} - \frac 1 {\rho_{0}}\p_x p(\rho_{0}) + \frac{\p_x \mu(\rho_{0})}{\rho_{0}} \p_x u_{0}+f,\\
	& u_{1}|_{t = 0} = u_0(x,0).		
	\end{aligned}
	\end{equation}
	
	Let now $n \geq 2$. Given $\rho_{n-1}, u_{n-1}$, we iteratively define $\rho_{n}$ first, and subsequently $u_{n}$ as follows
	\begin{align}
	&\p_t \rho_{n} + u_{n-1}\p_x \rho_{n} = - \rho_{n-1}\p_x u_{n-1},\label{eq:itma}\\
	&\p_t u_{n} - \frac{\mu(\rho_{n})}{\rho_{n}} \p_x^2 u_{n} = - u_{n-1} \p_x u_{n-1} - \frac 1 {\rho_{n-1}}\p_x p(\rho_{n-1}) + \frac{\p_x \mu(\rho_{n-1})}{\rho_{n-1}} \p_x u_{n-1}+f,\label{eq:itmo}\\
	&(\rho_{n}, u_{n})|_{t = 0} = (\rho_{0}, u_0).\label{eq:itin}
	\end{align}
	Let $k \geq 1$ be an integer. We let, for ease of notation,
	\begin{equation*}
	{A := \Vert \rho_0 \Vert_{H^k}+\Vert u_0\Vert_{H^k}.}
	\end{equation*}
	
	We are going to prove, by induction on $n$, that there exists $T_0 > 0$ such that the following assertions hold.
	\begin{enumerate}
		\item[\bf{Step 1}:] There exists $u_1 \in C^\infty(\T \times [0, T_0])$ satisfying~\eqref{eq:initu1} and 
		\begin{equation}\label{eq:hyp1}
		\Vert u_1\Vert_{L^\infty(0, T_0; H^k)} \leq 2A, \qquad \int_{0}^{T_0} \int_{\T} \frac{\mu(\rho_1)}{\rho_1} (\p^{k+1}_x u_1)^2\de x\de t \leq 8A.
		\end{equation}
		\item[\bf{Step 2}:] For $n\geq 2$, there exists $\rho_n \in C^\infty(\T \times [0, T_0])$ satisfying~\eqref{eq:itma},~\eqref{eq:itin}, and
		\begin{equation*}
		\rho_n(x,t) \geq \frac {r_0}2 \text{ on } \T \times [0,T_0].
		\end{equation*}
		Furthermore,
		\begin{equation*}
		\Vert \rho_n \Vert_{L^\infty(0, T_0; H^k)} \leq 2A.
		\end{equation*}
		\item[\bf{Step 3}:] There exists $u_n \in C^\infty(\T \times [0, T_0])$ satisfying~\eqref{eq:itmo},~\eqref{eq:itin}, and 
		\begin{equation*}
		\Vert u_n\Vert_{L^\infty(0, T_0; H^k)} \leq 2A, \qquad \int_{0}^{T_0} \int_{\T} \frac{\mu(\rho_n)}{\rho_n} (\p^{k+1}_x u_n)^2\de x\de t \leq 8A.
		\end{equation*}
		\item[\bf{Step 4}:] The sequence $(\rho_n, u_n)$ is Cauchy in the space $L^\infty(0,T_0; L^2) \times \big(L^\infty(0,T_0; L^2)\cap L^2(0,T_0; H^1)\big)$.	
		\item[{\bf Step 5}:] There exist 
		\[
		u\in C(0, T_0; H^{k}) \cap L^2(0, T_0; H^{k+1})
		\]
		and 
		\[
		\rho \in  C(0, T_0; H^{k})
		\] 
		such that $(\rho, u)$ is a strong solution to the system~\eqref{eq:mass}--\eqref{eq:mom} with initial data $(\rho_0, u_0)$. In particular, if $k =3$, said solution is a classical solution.
		\item[{\bf Step 6}:] The constructed strong solution is unique.
	\end{enumerate}
	
	Let us now turn to the details.
	
	{\bf Step 1.} This is the base case of the induction. The existence of $u_1$ in the conditions follows from the general theory of linear parabolic equations, using the fact that $\rho_0$ is bounded from below by $r_0$, and that all functions involved are smooth. The {bound \eqref{eq:hyp1}} is obtained exactly as in {\bf Step 3}, and we omit the details here.
	
	{\bf Step 2.} Let $n \geq 2$. Let us adopt the following nomenclature:
	\[
	\rho := \rho_{n}, \quad \eta:= \rho_{n-1}, \quad u := u_{n}, \quad v := u_{n-1}.
	\]
	We recall the induction hypotheses:
	\begin{equation}\label{eq:indhyp}
	\begin{aligned}
	&\Vert v\Vert_{L^\infty(0, T_0; H^k)} \leq 2A, & \Vert \eta \Vert_{L^\infty(0, T_0; H^k)} \leq 2A,\\
	&\int_{0}^{T_0} \int_{\T} \frac{\mu(\eta)}{\eta} (\p^{k+1}_x v)^2\de x\de t \leq 8A, & \inf_{t \in [0, T_0]}\inf_{x \in \T} \eta(x,t) \geq \frac{r_0} 2.
	\end{aligned}
	\end{equation}
	
	Existence up to time $T_0$ and smoothness for $\rho_n$ follow from the method of characteristics. 
	
	In what follows, $M(\cdot, \ldots, \cdot)$ will always denote a positive, continuous function increasing in all its arguments.  We first notice that, due to the mass equation~\eqref{eq:itma} and the maximum principle, for all $k \geq 1$ and $0~\leq~t~\leq~T_0$,
	\begin{equation}
	\inf_{\T} \rho(\cdot, t) \geq \inf_{\T} \rho_0 - \int_0^t \Vert \eta(\cdot, s) \p_x v(\cdot, s) \Vert_{L^\infty} \de s \geq \inf_{\T} \rho_0 -M(A) \sqrt t { \Vert\p_x^2 v\Vert_{L^2(0,t; L^2)}}.
	\end{equation}
	Hence, restricting $T_0$ to be small only as a function of $A$ and $r_0$, we have
	\[
	\inf_{t \in [0,T_0]} \inf_{x\in \T} \rho(x,t) \geq \frac{r_0}{2}.
	\]
	We have therefore recovered the last induction hypothesis in \eqref{eq:indhyp}.
	
	Let us now differentiate the mass equation~\eqref{eq:itma} $k$-times, multiply it by $\p^k_x \rho$ and integrate by parts
	\begin{equation}\label{eq:ref2}
	\frac 12 \p_t \int_\T (\p_x^k \rho )^2 \de x+ \int_\T \p_x^k \rho \, \p_x^k (v \p_x \rho) \de x= -\int_\T \p_x^k \rho \, \p_x^k (\eta \p_x v) . 
	\end{equation}
	If $k = 1$, we obtain
	\begin{align}
	&\frac 12 \p_t \Vert \rho \Vert^2_{L^2} \leq C \Vert \p_x^2 v \Vert_{L^2} \Vert \rho \Vert^2_{L^2} + \Vert \rho\Vert_{L^2} \Vert \eta \Vert_{L^\infty} \Vert \p_x v \Vert_{L^2}, \label{eq:rhozero}\\
	&\frac 12 \p_t \Vert \p_x \rho \Vert^2_{L^2} \leq C \Vert \p_x^2 v \Vert_{L^2} \Vert \p_x \rho \Vert^2_{L^2} + {2}\Vert \p_x \rho \Vert_{L^2} \Vert \p_x \eta \Vert_{L^2} \Vert \p_x v \Vert_{L^\infty} +\Vert \p_x \rho \Vert_{L^2} \Vert \eta \Vert_{L^\infty} \Vert \p^2_x v \Vert_{L^2}  \label{eq:rhoone}.
	\end{align}
	Combining~\eqref{eq:rhozero} and~\eqref{eq:rhoone}, integrating and using the induction hypotheses, we obtain, for suitable $T_0$ (depending only on $A$ and $r_0$)
	\begin{equation}
	\Vert{\rho}\Vert_{L^\infty(0, T_0; H^1)} \leq 2A.
	\end{equation}
	If $k \geq 2$, in addition to previous estimate~\eqref{eq:rhozero}, we also have, {for the terms appearing in~\eqref{eq:ref2},}
	\begin{equation}
	\begin{aligned}
	&\left|\int_\T \p_x^k \rho \, \p_x^k (v \p_x \rho) \de x\right| = \left|\frac 12 \int_\T v \p_x(\p_x^k \rho)^2 \de x + \int_\T \p^k_x \rho \, ([\p_x^k, v] \p_x\rho) \de x\right| \\
	&\leq \frac 12 \Vert \p_x v \Vert_{L^\infty} \Vert \rho \Vert^2_{H^k} + \Vert \rho \Vert_{H^k}\Vert [\p_x^k, v] \p_x\rho \Vert_{L^2} \leq C \Vert v \Vert_{H^2} \Vert \rho \Vert^2_{H^k} + C { \Vert  \rho \Vert^2_{H^k}}\Vert v\Vert_{H^k}.
	\end{aligned}
	\end{equation}
	Furthermore,
	\begin{equation}\label{eq:rhohigh}
	\begin{aligned}
	&\left| \int_\T \p_x^k \rho \, \p_x^k (\eta \p_x v)\right| \leq \Vert \rho\Vert_{H^k} \Vert \eta  \p_x^{k+1} v\Vert_{L^2} + \Vert \rho\Vert_{H^k} \Vert [\p_x^k, \eta] \p_x v \Vert_{L^2}\\
	&\leq C\Vert \rho\Vert_{H^k} \left(\left\Vert \frac{\eta^3}{\mu(\eta)} \right\Vert^{\frac 12}_{L^\infty} \left\Vert  \left(\frac{\mu(\eta)}{\eta} \right)^{\frac 12} \p_x^{k+1} v \right \Vert_{L^2} + \Vert  \eta \Vert_{H^2} \Vert  v\Vert_{H^k} + \Vert v \Vert_{H^2} \Vert  \eta \Vert_{H^k} \right).
	\end{aligned}
	\end{equation}
	Now, due to our assumptions on $\mu$ and the induction hypothesis, we have
	\begin{equation*}
	\left\Vert \frac{\eta^3}{\mu(\eta)} \right\Vert^{\frac 12}_{L^\infty} \leq M(A, r^{-1}_0),
	\end{equation*}
	where $M$ depends on $\mu$ and is an increasing function of its arguments.
	
	Upon summation of~\eqref{eq:rhozero}~and~\eqref{eq:ref2}, using~\eqref{eq:rhozero} and~\eqref{eq:rhohigh},
	\begin{equation*}\label{eq:gron1}
	\frac 12 \p_t \Vert \rho \Vert^2_{H^k} \leq C \Vert v \Vert_{H^k} \Vert \rho \Vert^2_{H^k} + C \Vert \rho \Vert_{H^k}\Vert \eta \Vert_{H^k}\Vert v \Vert_{H^k} + M(A, r_0^{-1}) \Vert \rho \Vert_{H^k}\left\Vert  \left(\frac{\mu(\eta)}{\eta} \right)^{\frac 12}\p_x^{k+1} v \right \Vert_{L^2}.
	\end{equation*}
	We now use the induction hypothesis~\eqref{eq:indhyp} to obtain, for $0 \leq t \leq T_0$,
	\begin{equation*}
	\p_t \left(\Vert \rho\Vert_{H^k} \exp\left( - 2 C At\right)\right)\leq 4 C A^2 + M(A, r_0^{-1})\left\Vert  \left(\frac{\mu(\eta)}{\eta} \right)^{\frac 12} \p_x^{k+1} v \right \Vert_{L^2}.
	\end{equation*}
	Upon integration, we obtain the following inequality:
	\begin{equation*}
	\Vert \rho \Vert_{H^k} \leq \exp\left(2CAt \right)\left( \Vert \rho_0 \Vert_{H^k} + 4 C A^2 t + 8 A \sqrt t M(A, r_0^{-1})  \right).
	\end{equation*}
	It is now straightforward to choose $T_0$, depending only on $A$ and $r_0$, such that the induction hypothesis
	\begin{equation*}
	\Vert \rho \Vert_{L^\infty(0, T_0; H^k)} \leq 2A
	\end{equation*}
	is recovered for $\rho$, in case $k \geq 2$.
	
	{\bf Step 3.} We now turn to the estimates on the momentum equation~\eqref{eq:itmo}.
	Multiplying such equation by $u$ and integrating by parts yields
	\begin{equation}\label{eq:momest0}
	\begin{aligned}
	\frac 12 \p_t \int_\T u^2 \de x - \int_\T  \frac{\mu(\rho)}{\rho}u \p_x^2 u  \de x = \int_{\T}u \cdot G_0 \, \de x,
	\end{aligned}
	\end{equation}
	where $G_0 := -v \p_x v - \frac 1 \eta \p_x p(\eta) + \frac{\p_x \mu(\eta)}{\eta}\p_x v+f$.
	If $k \geq 1$, this implies
	\begin{equation}\label{eq:uzero}
	\begin{aligned}
	&\frac 12 \p_t \Vert u \Vert^2_{L^2} + \int_\T \frac{\mu(\rho)}{\rho} (\p_x u)^2\de x \leq M(A, r_0^{-1}) \Vert \rho \Vert_{H^1} \Vert \p_x u \Vert_{L^2}\Vert u \Vert_{L^\infty}\\
	&\quad  + C \Vert u \Vert_{L^2}\Vert v \Vert^2_{H^1} + M(A, r_0^{-1}) (\Vert \eta \Vert_{H^1} \Vert u \Vert_{L^2} + \Vert \eta \Vert_{H^1} \Vert v \Vert_{H^1}\Vert u \Vert_{H^1} +  \Vert f \Vert_{L^2}\Vert u \Vert_{L^2}).
	\end{aligned}
	\end{equation}
	Here, we used integration by parts and the following Lemma
	
	\begin{lemm}\label{lem:para}
		Let $f$ be a smooth function away from $0$, and $k$ be a positive integer. Let $u \in H^k(\T)\cap L^\infty(\T)$, and suppose that there exists $r_0 >0$ such that $u \geq r_0$ on $\T$. Then, there exists a positive and continuous function $M$ which depends  only on $f$, $k$ and is increasing in both its arguments such that the following inequality holds:
		\begin{equation}
		\Vert f \circ u \Vert_{H^k(\T)} \leq M\big(\Vert u \Vert_{L^\infty(\T)}, r_0^{-1}\big) \Vert u\Vert_{H^k(\T)}.
		\end{equation}
	\end{lemm}
	\begin{proof}[Proof of Lemma~\ref{lem:para}]
		The proof of the lemma follows from Theorem 2.87 in~\cite{bcd}, \S 2.8.2, and a straightforward cutoff argument.
	\end{proof}
	\begin{rema}
		In what follows, we will always suppress the dependence of $M$ on $k$ and $f$, since they are fixed at the beginning of the argument.
	\end{rema}
	
	Differentiating $k$-times ($k \geq 1$) equation~\eqref{eq:itmo}, multiplying by $\p_x^k u$, and integrating by parts yields
	\begin{equation}\label{eq:momest}
	\begin{aligned}
	\frac 12 \p_t \int_\T (\p_x^k u)^2 \de x - \int_\T  (\p^k_x u) \p_x^k \left(\frac{\mu(\rho)}{\rho} \p_x^2 u \right) \de x = -\int_{\T}(\p_x^{k+1} u) \cdot G_k \, \de x.
	\end{aligned}
	\end{equation}
	Here, we defined
	\begin{equation*}
	G_k := \p_x^{k-1} \left(-v \p_x v - \frac 1 \eta \p_x p(\eta) + \frac{\p_x \mu(\eta)}{\eta}\p_x v+f\right), \quad \text{ for } k \geq 1.
	\end{equation*}
	When $k = 1$, the previous display~\eqref{eq:momest} implies, upon integration by parts, an application of the Cauchy--Schwarz inequality, the induction hypotheses, Lemma~\ref{lem:para} and the bounds obtained in {\bf Step 2}, that
	\begin{equation}\label{eq:momone}
	\begin{aligned}
	&\frac 12 \p_t \Vert \p_x u \Vert^2_{L^2} + \frac 12 \int_\T \frac{\mu(\rho)}{\rho} (\p^2_x u)^2\de x \leq \int_{\T} \frac{\rho}{\mu(\rho)}G_1^2 \de x\\
	& \quad \leq   M(A, r_0^{-1})(\Vert v \Vert^4_{H^1} + \Vert \eta \Vert^2_{H^1}+ \Vert \eta \Vert^2_{H^1} \Vert \p_x v \Vert_{L^2} \Vert \p^2_x v\Vert_{L^2} +  \Vert f \Vert^2_{L^2}).
	\end{aligned}
	\end{equation}
	
	Integrating~\eqref{eq:momone}~and, subsequently,~\eqref{eq:uzero}, upon restricting $T_0$ to be sufficiently small only as a function of $A$ and $r_0$, we have, in case $k=1$,
	\begin{equation*}
	\Vert u \Vert_{L^\infty(0,T_0; H^1)} \leq 2A, \qquad \int_0^{T_0} \frac{\mu(\rho)}{\rho} (\p^2_x u)^2 \de x \de t \leq 8A.
	\end{equation*}
	
	Let's focus now on the case $k \geq 2$. We have
	\begin{equation*}
	\begin{aligned}
	&-\int_\T  (\p^k_x u) \p_x^k \left(\frac{\mu(\rho)}{\rho} \p_x^2 u \right) \de x\\
	&= -\int_\T  (\p^k_x u) \p_x^{k+1} \left(\frac{\mu(\rho)}{\rho} \p_x u \right) \de x + \int_\T  (\p^k_x u) \p_x^{k} \left(\p_x \left(\frac{\mu(\rho)}{\rho}\right) \p_x u \right) \de x\\
	&= \int_\T  \frac{\mu(\rho)}{\rho} (\p_x^{k+1} u)^2\de x  \underbrace{+\int_\T \p_x^{k+1} u \, \left[\p_x^k , \frac{\mu(\rho)}{\rho}\right]  (\p_x u) \de x }_{(a)}\underbrace{-  \int_\T  (\p^{k+1}_x u) \p_x^{k-1} \left(\p_x \left(\frac{\mu(\rho)}{\rho}\right) \p_x u \right) \de x}_{(b)}.
	\end{aligned}
	\end{equation*}
	We estimate the last two terms in the previous display:
	\begin{equation}\label{eq:moma}
	\begin{aligned}
	|(a)|&\leq \frac 1 {10} \int_\T \frac{\mu(\rho)}{\rho} (\p_x^{k+1} u)^2 \de x + C \int_{\T}\frac{\rho}{\mu(\rho)} \left( \left[\p_x^k , \frac{\mu(\rho)}{\rho}\right]  (\p_x u) \right)^2 \de x\\
	&\leq \frac 1 {10} \int_\T \frac{\mu(\rho)}{\rho} (\p_x^{k+1} u)^2 \de x +M(A, r^{-1}_0)\left\Vert \left[\p_x^k , \frac{\mu(\rho)}{\rho}\right]  (\p_x u) \right \Vert_{L^2}\\
	&\leq \frac 1 {10} \int_\T \frac{\mu(\rho)}{\rho} (\p_x^{k+1} u)^2 \de x +M(A, r^{-1}_0) \left( \left \Vert \p_x \frac{\mu(\rho)}{\rho} \right \Vert_{L^\infty} \Vert \p_x^k u \Vert_{L^2} + \Vert \p_x u \Vert_{L^\infty}\left \Vert \p_x^k \frac{\mu(\rho)}{\rho} \right \Vert_{L^2} \right)\\
	& \leq \frac 1 {10} \int_\T \frac{\mu(\rho)}{\rho} (\p_x^{k+1} u)^2 \de x + M \left(A, r_0^{-1}\right)\Vert u\Vert_{H^k}.
	\end{aligned}
	\end{equation}
	Here, $M$ is a continuous and increasing function of its arguments. We used the bounds obtained in {\bf Step~2}, the Kato--Ponce commutator estimate, the fact that $k \geq 2$ and Lemma~\ref{lem:para} quoted below, applied to the function $\frac{\mu(\rho)}{\rho}$.

	Similarly, the following estimate holds true, for $k \geq 2$:
	\begin{equation}\label{eq:momb}
	\begin{aligned}
	|(b)| \leq  \frac 1 {10} \int_\T \frac{\mu(\rho)}{\rho} (\p_x^{k+1} u)^2 \de x + M \left(A, r_0^{-1}\right)\Vert u\Vert_{H^k}.
	\end{aligned}
	\end{equation}
	Again, $M$ is a positive, continuous and increasing function of its arguments.
	
	We now proceed to estimate the terms contained in the RHS of equation~\eqref{eq:momest} (the terms named ``$G$''), in case $k \geq 2$:
	\begin{equation*}
	\begin{aligned}
	\left| \int_{\T}(\p_x^{k+1} u) \cdot G_k \, \de x\right| \leq \frac 1 {10} \int_\T \frac{\mu(\rho)}{\rho} (\p_x^{k+1} u)^2 \de x + 5 \int_\T \frac{\rho}{\mu(\rho)} G_k^2 \de x.
	\end{aligned}	
	\end{equation*}
	Due to the bounds on $\rho$, we have
	\begin{equation*}
	\begin{aligned}
	\int_\T \frac{\rho}{\mu(\rho)} G_k^2 \de x \leq M\left(A, r_0^{-1}\right)\Vert G_k \Vert^2_{L^2}.
	\end{aligned}
	\end{equation*} 
	
	Let us now define two auxiliary functions ${h}$ (the thermodynamic enthalpy) and $\zeta$ in such a way that
	\begin{equation*}
	{h}'(x) = \frac{p'(x)}{x}, \qquad \zeta'(x) = \frac{\mu'(x)}{x}, \text{ for } x > 0.
	\end{equation*}
	We now estimate:
	\begin{equation*}
	\begin{aligned}
	\Vert \p_x^{k-1} (v \p_x v)\Vert^2_{L^2} \leq C \Vert  v \Vert^2_{{H^2}} \Vert v \Vert^2_{H^k} \leq CA^4.
	\end{aligned}
	\end{equation*}
	Furthermore,
	\begin{equation*}
	\begin{aligned}
	\Vert \p_x^{k-1} \left(\frac{\p_x p(\eta)}{\eta}\right) \Vert^2_{L^2} \leq \Vert {h}(\eta)\Vert^2_{H^k} \leq M(A, r_0^{-1}),
	\end{aligned}
	\end{equation*}
	where we used Lemma~\ref{lem:para}, applied to the function ${h}$.
	
	Finally, we have, since $k \geq 2$,
	\begin{equation*}
	\begin{aligned}
	&\left\Vert \p_x^{k-1} \left(\frac{\p_x \mu(\eta)}{\eta}\p_x v \right)\right\Vert^2_{L^2} = \Vert \p_x \zeta(\eta) \p_x v \Vert^2_{H^{k-1}} \leq C \left( \Vert \zeta(\eta) \Vert_{H^k} \Vert \p_x v \Vert_{L^\infty}+ \Vert v \Vert_{H^k} \Vert \p_x \zeta(\eta) \Vert_{L^\infty} \right)\\
	&\quad \leq M(A, r_0^{-1}).
	\end{aligned}
	\end{equation*}
	
	Hence, for the term $G_k$, we have
	\begin{equation}\label{eq:sumgk}
	\begin{aligned}
	\left| \int_{\T}(\p_x^{k+1} u) \cdot G_k \, \de x\right| \leq \frac 1 {10} \int_\T \frac{\mu(\rho)}{\rho} (\p_x^{k+1} u)^2 \de x + M(A, r_0^{-1}) {(1+\Vert f \Vert^2_{H^{k-1}})}.
	\end{aligned}
	\end{equation}
	
	Putting together estimates~\eqref{eq:momest0},~\eqref{eq:momest},~\eqref{eq:moma},~\eqref{eq:momb},~\eqref{eq:sumgk}, and ignoring the positive integral term in the LHS, we obtain the inequality
	\begin{equation*}
	\begin{aligned}
	\frac 1 2 \p_t \Vert u \Vert^2_{H^k} \leq M\left(A, r_0^{-1} \right) \Vert u \Vert_{H^k} + M\left(A, r_0^{-1}\right){(1+\Vert f \Vert^2_{H^{k-1}})}.
	\end{aligned}
	\end{equation*}
	Using Gr\"onwall's inequality, upon restricting $T_0$ to be small depending only on $A$, $r_0$ {and $f$}, we deduce that
	\begin{equation}
	\Vert u \Vert_{L^\infty(0,T_0; H^k)} \leq 2A.
	\end{equation}
	We now revisit the same estimates without discarding the positive integral term in the LHS. We obtain, upon restricting $T_0$ to be smaller, depending only on $A$ and $r_0$ {and $f$}, that
	\begin{equation}
	\int_0^{T_0} \int_{\T} \frac{\mu(\rho)}{\rho} (\p^{k+1}_x u)^2 \de x\de t \leq 8A.
	\end{equation}
	We have therefore recovered the induction hypotheses~\ref{eq:indhyp}, and in particular the sequence $(\rho_n, u_n)$ is uniformly bounded in $L^\infty(0,T_0; H^k(\T)) \times (L^\infty(0,T_0; H^k(\T)) \cap L^2(0,T_0; H^{k+1}(\T)) $.
	
	{\bf Step 4.} We now show that, for some $T_0$, depending only on $A, r_0$, the sequence $(\rho_{n}, u_{n})$ is Cauchy in the space $L^\infty(0,T_0; L^2) \times (L^\infty(0,T_0; L^2)\cap L^2(0, T_0; L^2))$.
	
	Let's first consider the equation satisfied by $\delta u_n:=u_{n+1}-u_n$:
	\begin{equation}\label{eq:diffu}
	\begin{aligned}
	&\p_t (\delta u_{n}) - \frac{\mu(\rho_{n+1})}{\rho_{n+1}} \p_x^2 u_{n+1} + \frac{\mu(\rho_{n})}{\rho_{n}} \p_x^2 u_{n}\\
	&= \frac 12 \p_x( u^2_{n}-u^2_{n-1}) + \p_x ({h}(\rho_{n})- {h}(\rho_{n-1}))+ \p_x \zeta(\rho_{n}) \, \p_x u_{n} - \p_x \zeta(\rho_{n-1})\, \p_x u_{n-1}.
	\end{aligned}
	\end{equation}
	Recall that we defined ${h}$ and $\zeta$ so that the following equalities hold true:
	\begin{equation*}
	\p_x {h}(\rho) = \frac{\p_x p(\rho)}{\rho}, \qquad \zeta(\rho) = \frac{\p_x \mu(\rho)}{\rho}.
	\end{equation*}
	We now multiply equation~\eqref{eq:diffu} by $\delta u_{n}$ and integrate by parts. We have:
	\begin{equation*}
	\begin{aligned}
	&\int_{\T } (\delta u_n)\left(- \frac{\mu(\rho_{n+1})}{\rho_{n+1}} \p_x^2 u_{n+1} + \frac{\mu(\rho_{n})}{\rho_{n}} \p_x^2 u_{n} \right)\de x\\
	&=\underbrace{- \int_{\T } (\delta u_n) \frac{\mu(\rho_{n+1})}{\rho_{n+1}} \p_x^2( \delta u_{n})\de x }_{(a)}+ \underbrace{\int_{\T} \left(\frac{\mu(\rho_{n})}{\rho_{n}} - \frac{\mu(\rho_{n+1})}{\rho_{n+1}}\right) \p_x^2 u_{n } (\delta u_n)\de x}_{(b)}.
	\end{aligned}
	\end{equation*}
	Note that, due to {\bf Step 3}, there exists $c = c(A, r_0)$ such that, up to time $T_0$, there holds $\frac{\mu(\rho_{i})}{\rho_{i}} \geq c$ for all integers $i \geq 0$.
	
	Hence, for the term in $(a)$, upon integration by parts,
	\begin{equation*}
	\begin{aligned}
	&(a) \geq c \Vert \p_x (\delta u_n)\Vert^2_{L^2}- \frac 1c \Vert \p_x \frac{\mu(\rho_{n})}{\rho_{n}}\Vert_{L^2} \Vert \delta u_n \Vert_{L^\infty}\Vert \p_x (\delta u_n) \Vert_{L^2} \\
	&\quad \geq  c \Vert \p_x (\delta u_n)\Vert^2_{L^2} - M(A, r^{-1}_0)\Big(\Vert \delta u_n \Vert^{\frac 12}_{L^2}\Vert \p_x (\delta u_n) \Vert^{\frac 32}_{L^2}+ \Vert \delta u_n \Vert_{L^2}\Vert \p_x (\delta u_n) \Vert_{L^2}\Big)\\
	&\quad \geq \frac c2 \Vert \p_x (\delta u_n)\Vert^2_{L^2}  - M(A, r_0^{-1})\Vert  \delta u_{n}\Vert^2_{L^2}.
	\end{aligned}
	\end{equation*}
	Here, we used Lemma~\ref{lem:para}, the Gagliardo--Nirenberg--Sobolev inequality and the Young inequality.
	
	We now estimate
	\begin{equation*}
	\begin{aligned}
	(b)\geq - M(A, r_0^{-1}) \Vert \delta \rho_{n}\Vert_{L^2} \Vert \p^2_x u_n \Vert_{L^2}\Vert \delta u_n\Vert_{L^2}^{\frac 12}\Vert \delta u_n\Vert_{H^1}^{\frac 12}.
	\end{aligned}
	\end{equation*}
	
	Let us now turn to the terms appearing in the RHS of~\eqref{eq:diffu}. We define
	\begin{equation*}
	\begin{aligned}
	&\underbrace{\int_{\T} \frac 12 \p_x( u^2_{n}-u^2_{n-1}) (\delta u_n)\de x}_{(c)} +\underbrace{ \int_{\T} (\delta u_n) \p_x (h(\rho_{n})- h(\rho_{n-1})) \de x}_{(d)}\\
	&\qquad \qquad\qquad\qquad +\underbrace{ \int_{\T} (\delta u_n) \left(\p_x \zeta(\rho_{n}) \, \p_x u_{n} - \p_x \zeta(\rho_{n-1})\, \p_x u_{n-1}\right) \de x}_{(e)}.
	\end{aligned}
	\end{equation*}
	Then, for $(c)$, we have, {after integration by parts,}
	\begin{equation*}
	\begin{aligned}
	&|(c)| \leq M(A) {\Vert \p_x( \delta u_{n}) \Vert_{L^2} \Vert \delta u_{n-1}\Vert_{L^2}}\leq \frac{1}{10c} \Vert \p_x(\delta u_n) \Vert^2_{L^2} + M(A){ \Vert \delta u_{n-1} \Vert_{L^2}^2}.
	\end{aligned}
	\end{equation*}
	Concerning the term $(d)$, instead,
	\begin{equation*}
	\begin{aligned}
	|(d)| &= \left| \int_{\T} \p_x (\delta u_n) \, ({h}(\rho_{n})- {h}(\rho_{n-1}))\de x \right| \leq \frac 1 {10c} \Vert \p_x(\delta u_n)\Vert^2_{L^2} + M( A,r_0^{-1}) \Vert \delta \rho_{n-1}\Vert_{L^2}^2.
	\end{aligned}
	\end{equation*}
	Again, we used the fact that, due to the uniform bounds on $\rho_{n}$, ${h}$ is Lipschitz of constant depending only on $A$ and $r_0$.
	
	Finally, concerning $(e)$,
	\begin{equation*}
	\begin{aligned}
	|(e)| &\leq \left|\int_\T (\delta u_n) \p_x \zeta(\rho_{n}) \p_x(\delta u_{n-1}) \de x \right| + \left|\int_\T (\delta u_n) \p_x (\zeta(\rho_{n})-\zeta(\rho_{n-1})) \p_x u_{n-1} \de x \right|\\
	&\leq \Vert \delta u_n \Vert_{L^\infty}\Vert \p_x \zeta(\rho_n) \Vert_{L^2}\Vert \p_x( \delta u_{n-1})\Vert_{L^2}
	+ \left|\int_\T  (\zeta(\rho_{n})-\zeta(\rho_{n-1})) \p_x((\delta u_n) \p_x u_{n-1} )\de x \right|\\
	&\leq  M(A, r_0^{-1}) \Big(\Vert \delta u_{n}\Vert^{\frac 12}_{L^2}\Vert \p_x(\delta u_{n})\Vert^{\frac 12}_{L^2} \Vert \p_x(\delta u_{n-1})\Vert_{L^2} +\Vert \p_x(\delta u_{n-1})\Vert_{L^2} \Vert \delta u_{n}\Vert_{L^2}\Big)\\
	&\quad + M(A, r_0^{-1})( \Vert \delta \rho_{n-1}\Vert_{L^2}\Vert \p_x\delta u_{n}\Vert_{L^2} \Vert \p_x^2 u_n \Vert^{\frac 12}_{L^2}+\Vert \delta \rho_{n-1}\Vert_{L^2}\Vert \delta u_{n}\Vert_{L^\infty}\Vert \p_x^2 u_n \Vert_{L^2})
	\end{aligned}
	\end{equation*}
	where $\delta \rho_{n-1}:=\rho_n-\rho_{n-1}$.
	Putting together the estimates on the momentum equation, we have
	\begin{equation*}
	\begin{aligned}
	&\frac 12 \p_t \Vert \delta u_n \Vert^2_{L^2} + \frac{1}{10c} \Vert \p_x (\delta u_n)\Vert^2_{L^2}  \\
	& \leq M(A, r_0^{-1})(\Vert  \delta u_{n}\Vert^2_{L^2} 
	+ \Vert \delta u_{n-1} \Vert_{L^2}^2+ \Vert \delta \rho_{n-1}\Vert_{L^2}^2)\\
	&\quad + M(A, r_0^{-1}) \Vert \delta \rho_{n}\Vert_{L^2} \Vert \p^2_x u_n \Vert_{L^2}\Vert \delta u_n\Vert_{L^2}^{\frac 12}\Vert \p_x(\delta u_n)\Vert_{L^2}^{\frac 12}\\
	&\quad + M(A, r_0^{-1}) (\Vert \delta u_{n}\Vert^{\frac 12}_{L^2}\Vert \p_x(\delta u_{n})\Vert^{\frac 12}_{L^2} \Vert \p_x(\delta u_{n-1})\Vert_{L^2} +\Vert \p_x(\delta u_{n-1})\Vert_{L^2} \Vert \delta u_{n}\Vert_{L^2})\\
	&\quad + M(A, r_0^{-1})( \Vert \delta \rho_{n}\Vert_{L^2}\Vert \p_x\delta u_{n}\Vert_{L^2} \Vert \p_x^2 u_n \Vert^{\frac 12}_{L^2}+\Vert \delta \rho_{n}\Vert_{L^2}\Vert \delta u_{n}\Vert_{L^\infty} \Vert \p_x^2 u_n \Vert_{L^2}).
	\end{aligned}
	\end{equation*}
	Upon integration between time $s = 0$ and $s= t$, using H\"older's inequality and the bounds obtained in {\bf Step~1},
	\begin{equation}\label{eq:int1}
	\begin{aligned}
	&\frac 12 \Vert (\delta u_n) (\cdot, t)\Vert^2_{L^2} + \frac 1{10c} \Vert \p_x(\delta u_n) \Vert^2_{L^2(0,t; L^2)}\\
	&\quad \leq M(A, r_0^{-1}) (\Vert  \delta u_{n}\Vert^2_{L^2(0,t; L^2)} 
	+ \Vert \delta u_{n-1} \Vert_{L^2(0,t; L^2)}^2+ \Vert \delta \rho_{n-1}\Vert_{L^2(0,t; L^2)}^2) \\
	&\quad + M(A, r_0^{-1}) t^{\frac 14} \Vert \delta \rho_{n} \Vert_{L^\infty(0, t; L^2)} \Vert \delta u_n \Vert^{\frac 12}_{L^\infty(0, t; L^2)} \Vert \p_x(\delta u_n)\Vert^{\frac 12}_{L^2(0,t; L^2)}\\
	&\quad+ M(A, r_0^{-1}) t^{\frac 1 4} \Vert \delta u_n \Vert^{\frac 12}_{L^\infty(0, t; L^2)} \Vert \p_x(\delta u_{n}) \Vert^{\frac 12}_{L^2(0, t; L^2)} \Vert \p_x(\delta u_{n-1})\Vert_{L^2(0,t; L^2)}\\
	&\quad + M(A, r_0^{-1}) t^{\frac 12}  \Vert \p_x(\delta u_{n-1})\Vert_{L^2(0,t; L^2)} \Vert \delta u_{n}\Vert_{L^\infty(0,t;L^2)}\\
	&\quad + M(A, r_0^{-1}) t^{\frac 14} \Vert \delta \rho_{n-1}\Vert_{L^\infty(0,t; L^2)}\Vert \p_x(\delta u_n)\Vert_{L^2(0,t; L^2)} \\
	&\quad + M(A, r_0^{-1})t^{\frac 14} \Vert\delta\rho_{n-1} \Vert_{L^\infty(0,t; L^2)}\Vert \p_x\delta u_n \Vert_{L^2(0,t; L^2)} \\	
	&\quad + M(A, r_0^{-1})t^{\frac 14} \Vert\delta\rho_{n-1} \Vert_{L^\infty(0,t; L^2)}\Vert \delta u_n \Vert^{\frac 12}_{L^\infty(0,t; L^2)} \Vert \p_x( \delta u_n )\Vert^{\frac 12}_{L^2(0,t; L^2)}  \\
	&\quad \leq \frac 1{20c} \Vert \p_x(\delta u_n) \Vert^2_{L^2(0,t; L^2)}+ M(A, r_0^{-1}) t^{\frac 14} (\Vert  \delta u_{n}\Vert^2_{L^\infty(0,t; L^2)} 
	+ \Vert \delta u_{n-1} \Vert_{L^\infty(0,t; L^2)}^2 +\\
	&\hspace{100pt}\Vert \delta \rho_{n-1}\Vert_{L^\infty(0,t; L^2)}^2 +  \Vert \p_x(\delta u_{n-1})\Vert^{2}_{L^2(0,t; L^2)}).
	\end{aligned}
	\end{equation}
	
	Let us now calculate the equation satisfied by differences of $\rho_n$:
	\begin{equation}\label{eq:diffr}
	\p_t (\delta \rho_{n}) = - u_n \p_x \rho_{n+1} + u_{n-1}\p_x \rho_n - \rho_n \p_x u_n + \rho_{n-1}\p_x u_{n-1}.
	\end{equation}
	Multiplying equation~\eqref{eq:diffr} by $\delta \rho_{n}$, we obtain
	\begin{equation*}
	\begin{aligned}
	\frac 12 \p_t \Vert\delta \rho_{n}\Vert^2_{L^2} =&-\underbrace{\int_\T (\delta \rho_{n})(u_{n}\p_x \rho_{n+1}-u_{n-1}\p_x \rho_{n})\de x}_{(a)}-\underbrace{\int_{\T} (\delta \rho_{n}) (\rho_{n} \p_x u_{n} - \rho_{n-1} \p_x u_{n-1})\de x}_{(b)}.
	\end{aligned}
	\end{equation*}
	Considering $(a)$, we have, integrating by parts, using Gagliardo--Nirenberg--Sobolev and H\"older's inequality,
	\begin{equation*}
	\begin{aligned}
	|(a)| &\leq \left| \int_{\T}(\delta \rho_{n})(\delta u_{n-1})\p_x \rho_{n+1} \de x\right| + \left| \int_{\T}\p_x(\delta \rho_{n})(\delta \rho_{n})u_{n-1} \de x\right| \\
	&\leq M(A) (\Vert \delta \rho_{n}\Vert_{L^2}\Vert \delta u_{n-1}\Vert^{\frac 12}_{H^1}\Vert \delta u_{n-1}\Vert^{\frac 12}_{L^2} + \Vert \delta \rho_{n}\Vert^2_{L^2} \Vert \p^2_x u_n \Vert^{\frac 12}_{L^2}).
	\end{aligned}
	\end{equation*}
	On the other hand, $(b)$ yields
	\begin{equation*}
	\begin{aligned}
	|(b)| &\leq \left| \int_{\T}(\delta \rho_{n})(\delta \rho_{n-1})\p_x u_{n} \de x\right| + \left| \int_{\T}(\delta \rho_{n}) \p_x(\delta u_{n-1})\rho_{n-1} \de x\right| \\
	&\leq M(A) (\Vert \delta \rho_{n}\Vert^2_{L^2} + \Vert \delta \rho_{n-1}\Vert^2_{L^2}) \Vert \p^2_x u_n \Vert^{\frac 12}_{L^2} +M(A) \Vert \p_x (\delta u_{n-1}) \Vert_{L^2} \Vert \delta \rho_{n} \Vert_{L^2}.
	\end{aligned}
	\end{equation*}
	Putting together the estimates on the mass equation yields
	\begin{equation*}
	\begin{aligned}
	&\frac 12 \p_t \Vert\delta \rho_{n}\Vert^2_{L^2}\\
	&\leq M(A) \Big(\Vert \delta \rho_{n}\Vert_{L^2}\Vert \p_x(\delta u_{n-1})\Vert^{\frac 12}_{L^2}\Vert \delta u_{n-1}\Vert^{\frac 12}_{L^2} + \Vert \delta \rho_{n}\Vert^2_{L^2} \Vert \p^2_x u_n \Vert^{\frac 12}_{L^2}\Big)+M(A) \Vert \delta \rho_{n}\Vert_{L^2}\Vert \delta u_{n-1}\Vert_{L^2}\\
	&\quad + M(A) (\Vert \delta \rho_{n}\Vert^2_{L^2} + \Vert \delta \rho_{n-1}\Vert^2_{L^2}) \Vert\p^2_x u_n\Vert^{\frac 12}_{L^2} +M(A) \Vert \p_x (\delta u_{n-1}) \Vert_{L^2} \Vert \delta \rho_{n} \Vert_{L^2}.
	\end{aligned}
	\end{equation*}
	Upon integration, the previous display yields
	\begin{equation}\label{eq:int2}
	\begin{aligned}
	\frac 12 \Vert \delta \rho_n(t, \cdot)\Vert^2_{L^2} &\leq M(A)t^{\frac 34} \Vert \delta \rho_{n}\Vert_{L^\infty(0,t; L^2)}\Vert \p_x(\delta u_{n-1})\Vert^{\frac 12}_{L^2(0,t;L^2)}\Vert \delta u_{n-1}\Vert^{\frac 12}_{L^\infty(0,t;L^2)} \\
	&\quad +M(A) t^{\frac 3 4}\Vert \delta \rho_{n} \Vert^2_{L^\infty(0,t; L^2)} + M(A) t (\Vert \delta \rho_{n}\Vert^2_{L^\infty(0,t; L^2)} + \Vert \delta u_{n-1}\Vert^2_{L^\infty(0,t; L^2)})\\ \nonumber
	& \quad+ M(A)t^{\frac 34} (\Vert \delta \rho_{n}\Vert^2_{L^\infty(0,t; L^2)} + \Vert \delta \rho_{n-1}\Vert^2_{L^\infty(0,t; L^2)}) \\
	&\quad+ M(A)t^{\frac 12} \Vert \p_x (\delta u_{n-1}) \Vert_{L^2(0,t;L^2)} \Vert \delta \rho_{n} \Vert_{L^\infty(0,t;L^2)}\\
	& \leq M(A)t^{\frac 12}\Big(\Vert \delta \rho_{n}\Vert^2_{L^\infty(0,t; L^2)}+ \Vert \p_x(\delta u_{n-1})\Vert^{2}_{L^2(0,t;L^2)}+\Vert \delta u_{n-1}\Vert^{2}_{L^\infty(0,t;L^2)} \\
	&\hspace{60pt} +\Vert \delta \rho_{n-1}\Vert^2_{L^\infty(0,t; L^2)}\Big).
	\end{aligned}
	\end{equation}
	Combining now~\eqref{eq:int1} and~\eqref{eq:int2}, we obtain, for suitably small $t$ depending only on $A$ and $r_0$,
	\begin{equation*}
	\begin{aligned}
	&\frac 14 \Vert \delta \rho_{n}\Vert^2_{L^\infty(0,t; L^2)} + \frac 1 4\Vert \delta u_{n}\Vert^2_{L^\infty(0,t; L^2)}+ \frac 1 {20c}\Vert \p_x( \delta u_{n}) \Vert^2_{L^2(0,t; L^2)}\\
	&\leq  M(A, r_0^{-1})t^{\frac 14}(\Vert \p_x(\delta u_{n-1})\Vert^{2}_{L^2(0,t;L^2)}+\Vert \delta u_{n-1}\Vert^{2}_{L^\infty(0,t;L^2)} +\Vert \delta \rho_{n-1}\Vert^2_{L^\infty(0,t; L^2)} ).
	\end{aligned}
	\end{equation*}
	
	Upon suitable choice of $T_0$, this implies that the sequence $(\rho_n, u_n)$ is Cauchy in the space $L^\infty(0, T_0; L^2)\times (L^\infty(0, T_0; L^2) \cap L^2(0, T_0; H^1) )$.
	
	{\bf Step 5}. Denote 
	\[
	X^m=L^\infty(0,T_0; H^m) \times \big(L^\infty(0,T_0; H^m)\cap L^2(0,T_0; H^{m+1})\big)
	\]
	a Banach space with its canonical norm. We have proved in the previous steps that $(\rho_n, u_n)$ is bounded in $X^k$ and Cauchy in $X^{k-1}$. The latter implies that  $(\rho_n, u_n)$ converges to some $(\rho, u)$ in $X^{k-1}$. The former implies that some subsequence $(\rho_{n_j}, u_{n_j})$ converges weak-* to some $(\rho_*, u_*)$ in $X^k$. Since both weak-* convergence in $X^k$ and strong convergence in $X^{k-1}$ imply convergence in the sense of distributions we deduce that $(\rho, u)=(\rho_*, u_*)\in X^k$. It can be easily verified that $(\rho, u)$ is a strong solution  to the system~\eqref{eq:mass}--\eqref{eq:mom}. Moreover, since $\rho_n\to \rho$ strongly in $L^2(0, T_0; L^2)$ and $(\rho_n)$ is bounded in $L^\infty(0, T_0; H^1)$ it follows by interpolation that $\rho_n\to \rho$ strongly in $L^\infty(0, T_0; H^{3/4})$, and hence in $L^\infty(0, T_0; L^\infty)$.  This combined with the fact that $\rho_n(x, t)\ge \frac{r_0}{2}$ for all $(x, t)\in \T\times [0, T_0]$ (see {\bf Step 2}) yields
	\[
	\rho(x, t)\ge \frac{r_0}{2}\quad \forall (x, t)\in \T\times [0, T_0].
	\]

	{\bf Step 6.} We now establish uniqueness of strong solutions. Consider solutions $(\rho_1, u_1)$ and $(\rho_2,u_2)$, such that
	\[
	\rho_i \in C(0, T_0; H^k(\T)),\quad u_i\in C(0, T_0; H^k(\T))\cap L^2(0, T_0; H^{k+1}(\T)), \text{ for } i =1,2.
	\] 
	and let $(\delta \rho,\delta u)= (\rho_1-\rho_2, u_1-u_2)$.  We have 
	\begin{align}
	& \partial_t  \delta u +  \delta u \partial_x  u_1 +  u_2 \partial_x  \delta u =- \partial_x ((\rho_1)- (\rho_2) ) + \rho_1^{-1}\partial_x(\mu(\rho_1) \partial_xu_1)- \rho_2^{-1}\partial_x(\mu(\rho_2) \partial_xu_2), \label{eq:uniqm}\\
	&\partial_t \delta \rho + \partial_x (u_1 \delta \rho + \rho_2 \delta u   ) = 0, \label{eq:uniqmass}\\
	&(\delta \rho,\delta u)|_{t = 0} = (0,0) 
	\end{align}
	We now notice that equation~\eqref{eq:uniqm} is the same as equation~\eqref{eq:diffu}, upon formally substituting $n = 1$ in the LHS, and $n =2$ in the RHS. Similarly, recalling~\eqref{eq:diffr}, we have
	\[
	\underbrace{\p_t (\delta \rho_{n})}_{(a)} = \underbrace{- u_n}_{(b)} \underbrace{\p_x \rho_{n+1}}_{(a)} + \underbrace{ u_{n-1}}_{(b)}\underbrace{\p_x \rho_n}_{(a)} \underbrace{- \rho_n \p_x u_n + \rho_{n-1}\p_x u_{n-1}}_{(b)}.
	\]
	Formally substituting $n = 1$ in terms $(a)$, and $n =2$ in terms $(b)$, we obtain~\eqref{eq:uniqmass}. It is then straightforward to see that the same estimates as in {\bf Step 4} yield uniqueness of strong solutions.
\end{proof}

\subsection*{Acknowledgment} We thank Toan Nguyen and the reviewers for interesting comments. The research of PC is partially supported by NSF grant DMS-1713985. The research of TD is partially supported by NSF grant DMS-1703997. The research of HN is partially supported by NSF grant DMS-1600028 and DMS-1265818.


\begin{thebibliography}{MWWZ00}

 \bibitem{perthame2001}
J.-F. Gerbeau and B.~Perthame.
\newblock Derivation of viscous {S}aint-{V}enant system for laminar shallow water; numerical validation.
\newblock {\em Discrete Contin. Dyn. Syst. Ser. B}, 1(1): 89--102, 2001.


\bibitem{Marche2007}
F. Marche.
\newblock Derivation of a new two-dimensional viscous shallow water model with
  varying topography, bottom friction and capillary effects.
\newblock {\em Eur. J. Mech. B Fluids}, 26(1):49--63, 2007.



\bibitem{JD94} J. Eggers and T. F. Dupont. Drop formation in a one-dimensional approximation of the Navier-Stokes equation. {\em J. Fluid Mech.} 262, 205--221, 1994.

\bibitem{JF15}  J. Eggers and M. A Fontelos. {\em Singularities: formation, structure, and propagation}. Cambridge Texts in Applied Mathematics. Cambridge University Press, Cambridge, 2015. xvi+453 pp. 



\bibitem{ED18}
G.L. Eyink and T.D. Drivas.: Cascades and dissipative anomalies in compressible fluid turbulence. Phys. Rev. X 8.1, 011022, 2018.


\bibitem{Hoff2001}
D. Hoff and J. Smoller.
\newblock Non-formation of vacuum states for compressible {N}avier-{S}tokes
  equations.
\newblock {\em Comm. Math. Phys.}, 216(2):255--276, 2001.

 \bibitem{AnTon} A. Bui.  
Existence and Uniqueness of a Classical Solution of an Initial-Boundary Value Problem of the Theory of Shallow Waters. {\em SIAM J. Math. Anal.}, 12(2), 229--241, 1981.

\bibitem{Li2017}
Y. Li, R. Pan, and S. Zhu.
\newblock On classical solutions to 2{D} shallow water equations with
  degenerate viscosities.
\newblock {\em J. Math. Fluid Mech.}, 19(1):151--190, 2017.

\bibitem{Matsumura1980}
A. Matsumura and T. Nishida.
\newblock The initial value problem for the equations of motion of viscous and
  heat-conductive gases.
\newblock {\em J. Math. Kyoto Univ.}, 20(1):67--104, 1980.

\bibitem{Matsumura1983}
A. Matsumura and T. Nishida.
\newblock Initial-boundary value problems for the equations of motion of
  compressible viscous and heat-conductive fluids.
\newblock {\em Comm. Math. Phys.}, 89(4):445--464, 1983.

 \bibitem{H14} 
B. Haspot. Existence of global strong solution for the compressible Navier-Stokes equations with degenerate viscosity coefficients in 1D. {\it Math. Nachr.}, Vol 291, Issue 14--15, 2188--2203, 2018.

\bibitem{MV06} A.~Mellet and A.~Vasseur.
\newblock Existence and uniqueness of global strong solutions for
  one-dimensional compressible {N}avier-{S}tokes equations.
\newblock {\em SIAM J. Math. Anal.}, 39(4):1344--1365, 2007/08.

\bibitem{Sundbye1996}
L. Sundbye.
\newblock Global existence for the {D}irichlet problem for the viscous shallow
  water equations.
\newblock {\em J. Math. Anal. Appl.}, 202(1):236--258, 1996.

\bibitem{Sundbye1998}
L. Sundbye.
\newblock Global existence for the {C}auchy problem for the viscous shallow
  water equations.
\newblock {\em Rocky Mountain J. Math.}, 28(3):1135--1152, 1998.



\bibitem{Kloeden1985}
P.~E. Kloeden.
\newblock Global existence of classical solutions in the dissipative shallow
  water equations.
\newblock {\em SIAM J. Math. Anal.}, 16(2):301--315, 1985.



\bibitem{CENV}
P. Constantin, T. Elgindi, H. Nguyen, and V. Vicol. On singularity formation in a Hele-Shaw model. {\it Communication in Mathematical Physics}, 363 (1), 139--171, 2018.
\bibitem{Hoff1995} D. Hoff.
Global solutions of the Navier-Stokes equations for multidimensional compressible flow with discontinuous initial data.  
\newblock{\em J. Differential Equations}, 120(1): 215--254, 1995.

\bibitem{Lions} P.-L. Lions. {\it Mathematical topics in fluid dynamics, Vol.2, Compressible
models.} Oxford Science Publication, Oxford, 1998.

\bibitem{Bresch2003}
D. Bresch and B. Desjardins.
\newblock {Existence of Global Weak Solutions for a 2D Viscous Shallow Water
  Equations and Convergence to the Quasi-Geostrophic Model}.
\newblock {\em Communications in Mathematical Physics}, 238(1), 211--223, 2003.

\bibitem{Bresch2003b}
D. Bresch, B. Desjardins, and  C. Lin.
\newblock On some compressible fluid models: {K}orteweg, lubrication, and
  shallow water systems.
\newblock {\em Comm. Partial Differential Equations}, 28(3-4):843--868, 2003.

\bibitem{BreschReview}
D. Bresch, B. Desjardins, and G. M\'etivier.
\newblock Recent mathematical results and open problems about shallow water
  equations.
\newblock In {\em Analysis and simulation of fluid dynamics}, Adv. Math. Fluid
  Mech., pages 15--31. Birkh\"auser, Basel, 2007.
  
\bibitem{bcd} H.~Bahouri, J.~-Y.~Chemin and R.~Danchin. {\em {F}ourier analysis and nonlinear partial differential equations}. Grundlehren der Mathematischen Wissenschaften 343. Springer, New York, 2011.

\bibitem{KP} 
T. Kato and G. Ponce. Commutator estimates and the Euler and Navier--Stokes equations. Communications on Pure and Applied Mathematics 41.7: 891-907, 1988.


\end{thebibliography}
\end{document}